\documentclass[a4paper,11pt]{amsart}
\usepackage{amssymb,amsfonts,amsxtra, mathrsfs,placeins,graphicx,verbatim}
\usepackage[all]{xy}
\xyoption{line}
\usepackage{fullpage}
\usepackage{euscript}
\newtheorem{theorem}{Theorem}[section]
\newtheorem{cor}[theorem]{Corollary}
\newtheorem{lem}[theorem]{Lemma}
\newtheorem{prop}[theorem]{Proposition}

\theoremstyle{definition}

\newtheorem{example}[theorem]{Example}
\newtheorem{defi}[theorem]{Definition}
\newtheorem{rem}[theorem]{Remark}

\numberwithin{equation}{section}
\DeclareMathOperator{\TOP}{TOP}
\DeclareMathOperator{\Iso}{Iso}

\DeclareMathOperator{\Hom}{Hom}
\DeclareMathOperator{\Aut}{Aut}
\DeclareMathOperator{\Ve}{Vert}
\DeclareMathOperator{\dgVect}{dgVect}
\DeclareMathOperator{\Edge}{Edge}
\DeclareMathOperator{\im}{Im}
\DeclareMathOperator{\Ker}{Ker}
\DeclareMathOperator{\Det}{Det}

\DeclareMathOperator*{\colim}{colim}
\DeclareMathOperator{\TFT}{TFT}
\DeclareMathOperator{\OTFT}{OTFT}
\DeclareMathOperator{\st}{st}

\newcommand{\V}{{\EuScript V}}

\def\G{\mathcal G}

\def\M{\mathscr M}
\def\MM{\mathcal M}

\def\C{\mathcal C}
\def\De{\Det^{d}}

\def\Mo{{\mathscr M}}

\def\End{\operatorname{End}}

\def\ra{\rightarrow}

\def\Z{\mathbb{Z}}
\def\S{\mathbb{S}}
\def\M{\mathbb{M}}
\def\D{\mathfrak{D}}
\def\ground{\mathbf k}
\def\K{\mathsf K}
\def\Ga{\overline{\Gamma}}

\newcommand{\noproof}{\begin{flushright}\ensuremath{\square}\end{flushright}}

\newcommand{\Ass}{\mathcal{A}\!\mathit{ss}}
\newcommand{\Comm}{\mathcal{C}\!\mathit{om}}
\newcommand{\Lie}{\mathcal{L}\!\mathit{ie}}

\newcommand{\Leg}{\operatorname{Leg}}
\newcommand{\BV}{\mathsf{BV}}
\newcommand{\bv}{\mathsf{bv}}
\newcommand{\Flag}{\operatorname{Flag}}

\def\O{\mathcal O}
\def\F{\mathsf F}
\newcommand\ft[2]{\F_{#1}{#2}}
\newcommand\dft[2]{\F_{#1}^{\vee}{#2}}

\def\Fo{\ft{}{\O}}
\def\Fvo{\dft{}{\O}}
\def\FDo{\ft{\D}{\O}}
\def\FvDo{\dft{\D}{\O}}
\def\FDeo{\ft{\De}{\O}}
\def\FvDeo{\dft{\De}{\O}}

\def\BVDo{\BV_{\D}\O}
\def\BVDeo{\BV_{\De}\O}
\def\BVDotemp{\BV^{\prime}_{\D}\O}

\def\bvDo{\bv_{\D}\O}
\def\bvDeo{\bv_{\De}\O}

\def\gd{\gamma}
\def\bd{\beta}

\def\emo{\O}

\newcommand\mco[2]{\overline{#1}^{#2}}
\newcommand\mcu[2]{\underline{#1}^{#2}}

\def\Cyc{\O}
\def\Cycgen{C}

\def\Cyco{\mco{\Cyc}{}}
\def\Cycu{\mcu{\Cyc}{}}
\def\Cycod{\mco{\Cyc}{d}}
\def\Cycud{\mcu{\Cyc}{d}}

\def\Commo{\mco{\Comm}{}}
\def\Commu{\mcu{\Comm}{}}
\def\Commod{\mco{\Comm}{d}}
\def\Commud{\mcu{\Comm}{d}}

\def\Asso{\mco{\Ass}{}}
\def\Assu{\mcu{\Ass}{}}
\def\Assod{\mco{\Ass}{d}}
\def\Assud{\mcu{\Ass}{d}}

\def\Assuo{\mcu{\Ass}{1}}

\def\KAss{\mathsf{K}\Ass}

\def\KAssd{\KAss^d}

\def\KAsso{\KAss^1}

\def\RAssd{\mathsf{R}\Ass^d}
\def\KRAssd{\mathsf{KR}\Ass^d}
\def\SAssd{\mathsf{S}\Ass^d}

\def\ArcCat{\mathcal{A}}
\def\Arc{\operatorname{Arc}}

\thanks{This work was completed during the first author's stay at IHES and he wishes to express his gratitude to this institution
for excellent working conditions. The second author thanks the EPSRC for its support.}

\begin{document}

\title[Dual Feynman transform...]{Dual Feynman transform for modular operads}
\author{J. Chuang \and A.~Lazarev}
\address{IH\'ES, Le Bois-Marie, 35 route de Chartres, F-91440, Bures-sur-Yvette,
France.}
\email{lazarev@ihes.fr}
\address{School of Mathematics\\University of Bristol\\Bristol BS8 1TW\\UK}
\email{joseph.chuang@bristol.ac.uk}

\begin{abstract}
We introduce and study the notion of a dual Feynman transform of a modular operad. This generalizes
and gives a conceptual explanation of Kontsevich's dual construction producing graph cohomology
classes from a contractible differential graded Frobenius algebra. The dual Feynman transform of a modular operad is indeed
linear dual to the Feynman transform introduced by Getzler and Kapranov when evaluated on vacuum graphs.
In marked contrast to the Feynman transform, the dual notion admits an extremely simple presentation
via generators and relations; this leads to an explicit and easy description of its algebras.
We discuss a further generalization of the dual Feynman transform whose algebras are not necessarily contractible.
This naturally gives rise to a two-colored graph complex analogous to the Boardman-Vogt topological tree complex.
\end{abstract}

\maketitle
\tableofcontents

\section*{Introduction}
The relationship between operadic algebras and various moduli spaces goes back to Kontsevich's seminal papers \cite{kontfeynman} and \cite{kontsg} where graph homology was also introduced.

Kontsevich proposed two constructions producing classes in the ribbon graph complex. The `direct' construction has as the input a $\Z/2$-graded $A_\infty$-algebra with an invariant scalar product (or symplectic $A_\infty$-algebra in the terminology of \cite{HamLaz1}). The output is a collection of homology classes in the ribbon graph complex (or a corresponding collection of cohomology classes in the moduli spaces of Riemann surfaces with marked points). This construction is now well understood from various standpoints, see, e.g. \cite{HamLaz2}, \cite{Costel}.

The other, `dual', construction starts with a contractible differential graded Frobenius algebra with an odd scalar product and gives rise to a collection of cohomology classes in the ribbon graph complex.
By pairing the direct and the dual constructions it is possible to prove the non-triviality of both. Explicit examples involving  so-called Moore algebras, cf. \cite{HamLaz1}, \cite{lazmod}, \cite{hamilt} will be discussed in future work.

The dual construction was motivated by the combinatorics of Feynman graphs in the quantum Chern-Simons theory. Its conceptual formulation based on the Batalin-Vilkovisky formalism is given in \cite{HamLaz3} (for the commutative graph complex). It is probable that the approach of \cite{HamLaz3} could be generalized to accommodate the case of general modular operads, using the language of operadic noncommutative geometry as it is presented in e.g. \cite{Gin}.

The purpose of this paper is to give the most general formulation of Kontsevich's dual construction  from the point of view of modular operads.  We introduce the notion of a `dual Feynman transform' $\Fvo$ for a differential graded modular operad $\O$.  It turns out that $\Fvo$ is itself a modular operad (more precisely, a mild generalization thereof which we call an \emph{extended modular operad}). In other words, there is no twisting as in the case of the usual or `direct' Feynman transform of \cite{GeK}. Moreover, the vacuum part of $\Fvo$, corresponding to graphs without legs, is indeed
linear dual
to $\Fo$. On the other hand the components corresponding to graphs with legs are contractible, in contrast to $\Fo$. Another striking difference between the direct and dual notions is that the dual Feynman transform $\Fvo$ admits an extremely simple presentation in terms of generators and relations; it is obtained from $\Fo$ by freely adjoining a generator $s$ of square zero and such that $d(s)=\mathbf 1$ where $\mathbf 1$ is the operadic unit. This leads to the correspondingly simple description of algebras over $\Fvo$ as $\O$-algebras with a choice of a contracting homotopy. Note that the description of algebras over $\Fo$ is much more complicated; it is given in terms of a quantum master equation, cf. \cite{bar}.

Applying this construction to the modular operad $\OTFT$ (open topological field theory) whose algebras are noncommutative Frobenius algebras we recover Kontsevich's dual construction. In fact we find that its original formulation needs to be modified in order
to produce cohomology classes. This can be achieved by imposing a rather stringent condition on the input Frobenius algebra. Alternatively one can modify the ribbon graph complex by allowing the contraction of certain (or all) loops. The latter approach corresponds to compactifying the moduli space of metric ribbon graphs, and the dual construction in fact produces cohomology classes of certain constructible sheaves on these compactifications.

We also introduce a certain generalization of the dual Feynman transform for a modular operad $\O$ which we call the Boardman-Vogt (or BV) resolution as suggested by Sasha Voronov. The BV-resolution is closely related to the canonical free resolution of $\O$ given by the twice-iterated Feynman transform of $\O$; it could be used to construct a minimal model of an $\O$-algebra. An algebra over the BV-resolution of $\O$ is an $\O$-algebra supplied with a Hodge-like decomposition compatible with the action of $\O$. Note that the de Rham algebra of a smooth oriented manifold does indeed have such a structure as follows from Hodge theory. This suggests a possible application to Chern-Simons theory on general manifolds.

\subsection*{Notation and conventions}
In this  paper we work mostly in the category of  $\Z/2$-graded
 vector spaces (also known as super-vector spaces) over a field $\ground$ of characteristic zero.
The reason for this choice is that operadic algebras  giving rise to nontrivial classes in graph complexes tend to be $\Z/2$-graded. However all our results (with obvious modifications) continue to hold in the $\Z$-graded context.
The adjective `differential graded' will mean `differential $\Z/2$-graded' and will be abbreviated as `dg'. The category of dg vector spaces over $\ground$ will be denoted by $\dgVect$. All of our unmarked tensors are understood to be taken over $\ground$.
On the other hand, by a \emph{chain complex} we will mean the usual $\Z$-graded notion -- a collection $C_\bullet=\{C_n\},n\in\Z$ of $\ground$-vector spaces together with a differential $d:C_n\ra C_{n-1}$. Similarly a \emph{cochain complex} will mean a collection $C^\bullet=\{C^n\},n\in\Z$  of $\ground$-vector spaces together with a differential $d:C^n\ra C^{n+1}$.
  The suspension of a chain complex  $C_\bullet$ is defined by $(\Sigma C)_i=C_{i-1}$ and that of a cochain complex $C^\bullet$ is defined by $(\Sigma C)^i=C^{i+1}$.
 For a $\Z/2$-graded vector space $V=V_0\oplus V_1$ the symbol $\Pi V$ will denote the \emph{parity reversion} of $V$; thus $(\Pi V)_0=V_1$ while $(\Pi V)_1=V_0$. For $d\in \Z/2$ we set $\Pi^dV:=\begin{cases}\Pi V \text{ if } d=1\\V \text{ if } d=0\end{cases}$. For an ungraded vector space $V$ of dimension $n$
and $d\in \Z/2$
we write $\De(V)$ for the $\Z/2$-graded vector space $(\Pi^n\Lambda^n V)^{\otimes d}\cong(\Pi V)^{\otimes nd}$. For a finite set $S$ we write $\De(S)$ for
$\De(\ground^S)$. Note that $\De(S)^*$ is canonically
isomorphic to $\De(S)$.

We use extensively the results of \cite{GeK}. Here is a brief summary of the relevant terminology (adapted to the $\Z/2$-graded context); the reader is referred to the original paper of Getzler and Kapranov's for details.
An ordered pair $(g,n)$ of nonnegative integers is \emph{stable} if $2(g-1)+n> 0$. The only pairs which are \emph{unstable}, i.e. not stable, are $(0,0)$, $(0,1)$, $(0,2)$ and $(1,0)$.
 The group $\Aut(1,2,\ldots,n)$ is denoted by $\S_n$.
 A \emph{cyclic} $\S$-module is a collection of dg vector spaces
 $\{\V((n))\left|\right. n\geq 0\}$ with an action of $\S_n$ on $V((n))$.
 A \emph{stable} $\S$-module is a collection of dg vector spaces $\{\V((g,n))\left|\right. g,n\geq 0\}$ with an action of $\S_n$ on each $\V((g,n))$, such that $\V((g,n))=0$ for unstable $(g,n)$.
   Furthermore, if $\V$ is a stable $\S$-module and $I$ is a finite set then we set
\[\V((g,I)):=\Bigl[\bigoplus\V((g,n))\Bigr]_{\S_n}\]
where the direct sum is extended over all bijections $\{1,2,\ldots,n\}\ra I$. We also set
$\V((n)):=\bigoplus_{g\geq 0}\V((g,n))$ and $\V((I)):=\bigoplus_{g\geq 0}\V((g,I))$.
We will consider the following stable $\S$-modules:
\[{\mathfrak s}((g,n))=\Pi^{n}\epsilon_n,\]
\[\Pi((g,n))=\Pi\ground.\]
Here $\epsilon _n$ is the alternating character of $\S_n$.

A \emph{graph} is a one-dimensional cell complex; we will only consider connected graphs. From a combinatorial perspective a graph consists of the following data:
\begin{enumerate}
\item
A finite set, also denoted by $G$, consisting of the \emph{half-edges} of $G$.
\item
A partition $\Ve(G)$ of $\Flag(G)$. The elements of $\Ve(G)$ are called the \emph{vertices} of $G$. The half-edges belonging to the same vertex $v$ will be denoted by $\Flag(v)$; the cardinality of $\Flag(v)$ will be called the \emph{valence} of $v$ and denoted by $n(v)$.
\item
An involution $\sigma$ acting on $\Flag(G)$. The \emph{edges} of $G$ are pairs of half-edges of $G$ forming a two-cycle of $\sigma$; the set of edges will be denoted by $\Edge(G)$. The \emph{legs} of $G$ are fixed points of $\sigma$; the set of legs will be denoted by $\Leg(G)$.\end{enumerate}

A \emph{stable} graph is a graph $G$ having each vertex $v\in \Ve(G)$ decorated by a non-negative integer $g(v)$, the \emph{genus} of $v$; it is required that $(g(v),n(v))$ is stable for each vertex $v$ of $G$. The \emph{genus} $g(G)$ of stable graph $G$ is defined by the formula
\begin{equation}\label{gen} g(G)=
\dim(H_1(G))
 + \sum_{v\in\Ve(G)}g(v).\end{equation}

Contracting an edge $e$ in a stable graph $G$ yields a new stable graph $G_e$; the decorations $g(v)$ of the vertices of $G_e$ are defined in the obvious way, so that $g(G_e)=g(G)$.  The set of stable graphs forms a category whose morphisms are generated by isomorphisms and edge-contractions $G\ra G_e$. We denote by $\Gamma((g,n))$ the category whose objects are stable graphs $G$ of genus $g$ with $n$ labeled legs (i.e., equipped with a bijection between
$\{1,\ldots,n\}$ and $\Leg(G)$). Note that $\Gamma((g,n))$ is empty whenever $(g,n)$ is unstable.

For a stable graph $G$ and a stable $\S$-module $\V$ we define the dg vector space
$$\V((G)):= \bigotimes_{v\in Vert(G)}\V((g(v),\Flag(v))$$
of $\V$ decorations on $G$.

Further, set
\[{\mathfrak K}(G)=\Det(\Edge(G));\]
\[\Det^d(G)=\Det^d(H_1(G)).\]

For a category $\C$, we denote by
$[\C]$ the set of isomorphism classes of objects, and by
 $\Iso \C$
the subcategory of isomorphisms.

\section{Main construction}\label{maincon}
In this section we introduce the notion of the dual Feynman transform of a modular operad and give its graphical interpretation. Crucially the dual Feynman transform contains an `unstable' generator; it will be necessary to relax the stability condition of Getzler and Kapranov.
\begin{defi}\
\begin{enumerate}\item An extended stable $\S$-module is a collection of dg vector spaces
$\{\V((g,n))|g,n\geq 0\}$ with an action of $\S_n$ on each
$\V((g,n))$, such that $\V((g,n))=0$ for $(g,n)=(0,0), (0,1), (1,0)$.
\item An extended stable graph is a connected graph $G$ having each vertex $v$ labeled by
a non-negative integer $g(v)$ called the \emph{genus} of $v$.
 We require that vertices of valence $0$ have genus at least $2$ and those of valence $1$ have genus at least $1$.
\end{enumerate}
\end{defi}
\begin{rem}
Our definition is a mild generalization of the notions of a stable $\S$-module and a stable graph, see \cite{GeK},
in that the unstable pair
$(0,2)$ is allowed.  In what follows we introduce analogues of the notions considered by \cite{GeK}
in the context of extended stable graphs and extended $\S$-modules. We will omit most of the details since the treatment
given by Getzler and Kapranov can be carried  over to the present situation almost verbatim.
\end{rem}
The \emph{genus} $g(G)$ of an extended stable graph $G$ is defined by the same formula  (\ref{gen}) as in the stable case.
Furthermore, the set of extended stable graphs of genus $g$ with $n$ labeled legs forms a category, just as
in the case of stable graphs; the morphisms are again generated by isomorphisms of graphs and edge-contractions. We
will denote this category by $\Ga((g,n))$.
Note that the only graphs in $\Ga((g,n))$ for unstable $(g,n)$ are the
strings
$\begin{xymatrix}  @C=2ex@R=2ex@M=-.1EX{
\ar@{-}[r]& \bullet \ar@{-}[rr] && \bullet \ar@{-}[r]&}\end{xymatrix}
\ldots
\begin{xymatrix}  @C=2ex@R=2ex@M=-.1EX{
\ar@{-}[r]& \bullet \ar@{-}[rr] && \bullet \ar@{-}[r]&}\end{xymatrix}
$ of bivalent vertices of genus $0$, which lie in $\Ga((0,2))$.

Suppose $(g,n)$ is stable. Then any extended stable graph
$G\in\Ga((g,n))$ contains at least one
 \emph{stable} vertex, i.e. a vertex $v$ for which $(g(v),n(v))$ is stable.  So $G$ determines a stable graph $\tilde{G}$ obtained from $G$
by forgetting the bivalent vertices of $G$ having genus~$0$. Clearly the association $G\mapsto \tilde{G}$ is a functor
$\Ga((g,n))\ra\Gamma((g,n))$.

Recall that a \emph{hyperoperad}
is a collection~$\D$ of functors from $\Iso\Gamma((g,n))$ to the category $\dgVect$
together with structure maps defined for any map between two stable graphs $f:G_0\ra G_1$:
\[\nu_f:\D(G_1)\otimes\bigotimes_{v\in\Ve(G_1)}\D(f^{-1}(v))\ra\D(G_0)\]
 subject to some natural coherence conditions.

 Any hyperoperad extends to a collection of functors on the categories
$\Iso\Ga((g,n))$: for an extended stable graph $G$ we set $\D(G):=\D(\tilde{G})$ if $(g,n)$ is stable and $\D(G):=\ground$ (with trivial $\S_n$-action) otherwise. The structure maps also extend in an obvious way: if the morphism~$f$ corresponds to contracting
an edge abutting an unstable vertex, i.e. a bivalent vertex of genus~$0$, then the corresponding $\nu_f$ is the identity map.


Given a hyperoperad $\D$ we define an endofunctor $\overline{\M}_\D$ on the category of extended stable $\S$-modules by the formula
\[\overline{\M}_\D\V((g,n))=\colim_{G\in \Iso\Ga((g,n))}\D(G)\otimes\V((G)).\]
Here $\V((G))$ is defined in the same way as in the stable case, and we refer to $\D(G)\otimes\V((G))$ as the space of $\D$-twisted $\V$-decorations on $G$.

The arguments of \cite{GeK} show that the functor $\overline{\M}_\D$ is a triple.

\begin{defi}
An extended modular $\D$-operad is an algebra over the triple $\overline{\M}_\D$. For an extended stable $\S$-module $\V$ the corresponding extended stable  $\S$-module $\overline{\M}_\D\V$ is called
the free extended modular $\D$-operad generated by $\V$.
\end{defi}

If $\V$ is stable the colimit defining
$\overline{\M}_\D\V((g,n))$ may be taken over $\Gamma((g,n))$,
and we recover Getzler and Kapranov's triple ${\M}_\D$ in the category of stable $\S$-modules. Thus a modular $\D$-operad is precisely an extended modular $\D$-operad whose underlying $\S$-module is stable.

Now let $\EuScript{S}$ be the extended stable $\S$-module for which $\EuScript{S}((g,n))=0$ if $(g,n)\neq (0,2)$ and $\EuScript{S}((0,2))=\Pi^l\ground$, a shifted copy of the ground field with the trivial action of $\S_2$.
We will denote by $s=\Pi^l1$ the canonical basis element in $\EuScript{S}((0,2))$.
Let $\O$ be a modular $\D$-operad and consider the coproduct of $\O$ and $\overline{\M}_\D\EuScript{S}$ in the category of extended modular $\D$-operads. Informally speaking, this is the $\D$-operad over $\O$ \emph{freely generated} by the operation $s$. To emphasize this point of view, we will denote
this operad by $\O[s]$.
 We have the following simple but crucial result.
\begin{lem}\label{key}
There is an isomorphism \[\O[s]((g,n))\cong \colim_{G\in \Iso\Ga((g,n))}\Det^l(\Edge(G))\otimes\D({G})\otimes\O((\tilde{G})).\] (If $(g,n)$ is unstable then $\tilde{G}$ is undefined and we read $\O((\tilde{G}))$ as $\ground$.)
\end{lem}
In other words, this formula expresses $\O[s]((g,n))$ as a space of $\D$-twisted $\O$-decorated extended stable graphs where the decorations are only placed
on the \emph{stable vertices}. The degrees of the graphs are then calculated according to the number of edges and the degree of $s$.
\begin{proof}
Let us assume first that $\O={M}_\D\V$ for some stable $\S$-module $\V$.
It is clear that
\begin{equation}\label{extended}\O[s]((g,n))=
{\M}_\D(\V\oplus\EuScript{S})((g,n))=\bigoplus_{G\in [\Iso \Ga((g,n))]}[\D(G)\otimes(\V\oplus\EuScript{S})(({G}))]_{\Aut(G)}.\end{equation}
Now for any extended stable graph $G$ we denote by $G^\prime$ the extended stable graph obtained from $G$ by contracting all edges connecting stable vertices (including loops at stable vertices)
and replacing all strings of unstable bivalent vertices of length $m$ by strings of length $m-1$.

\begin{figure}[h]
\[
\includegraphics[height=3.5in]{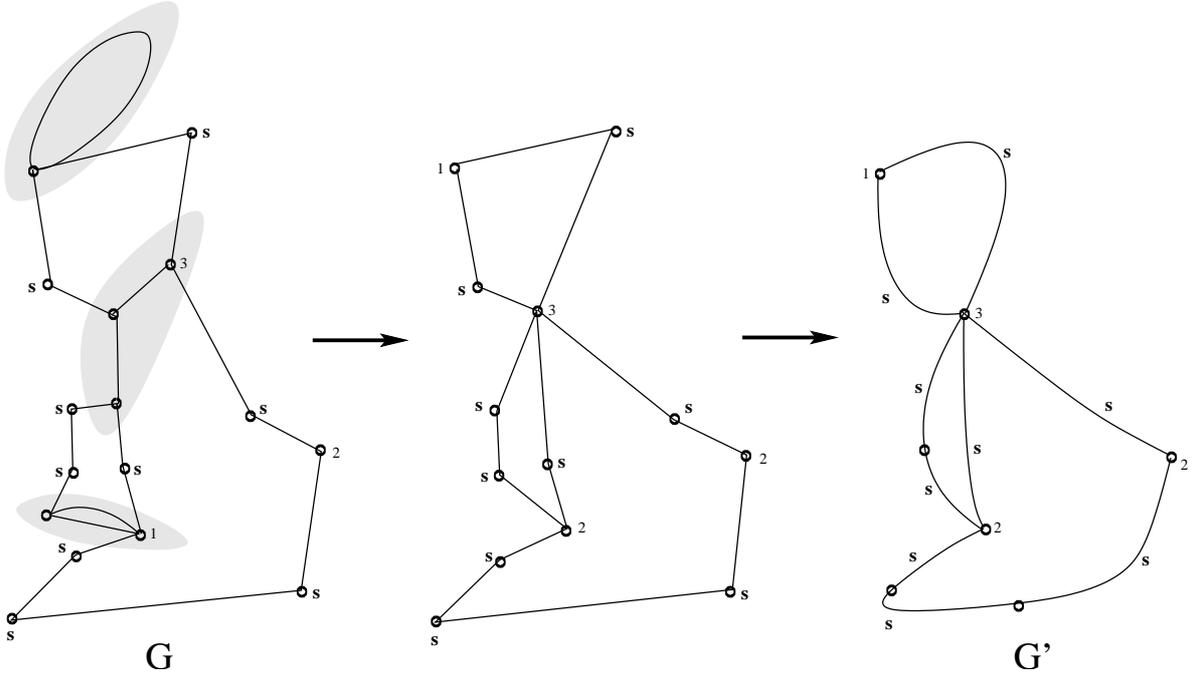}
\]
\caption{Contract edges connecting stable vertices, then replace unstable vertices by edges. Vertices are understood to have genus $0$ unless otherwise notated.\label{GtoGprime}}
\end{figure}
\FloatBarrier

Further, for any given $G^\prime\in\Ga((g,n))$ consider the set of all associated $G$'s and collect the corresponding terms in (\ref{extended}). Clearly, for a given $G^\prime$ we obtain
the $\D$-twisted space of $\O$-decorations on $G^\prime$, tensored with $\bigotimes \Pi^l \ground$, where the product is
taken over the unstable vertices of $G$ or equivalently the \emph{edges} of $G'$.  Here we are implicitly making use of the induced homomorphisms $\Aut(G)\ra\Aut(G')$.
The desired statement is proved in
this special case.

In general we have a canonical split coequalizer in the category of modular $\D$-operads
\[\xymatrix{{\M}_\D{\M}_\D\O\ar@<0.5ex>[r]\ar@<-0.5ex>[r]&{\M}_\D\O\ar[r]&\O},\]
and using the fact that a split coequalizer is preserved by any functor we reduce the general case to the one considered.

\end{proof}

\begin{rem}
Although our proof of Lemma \ref{key} is fairly unilluminating its statement is easy to visualize. The vertices of our decorated graphs
correspond to the operations in the operad $\O$ whilst the edges correspond to the adjoined operation $s$ so a decorated graph itself represents
an iterated composition of elements in $\O$ with $s$.
\end{rem}

\begin{defi}
Let $\D$ be a hyperoperad and $\O$ be a modular $\D$-operad. Define an extended stable $\S$-module  $\O_+$ by the formula
$\O_+((g,n))=\begin{cases} \O((g,n)) & \text{if $(g,n)\neq (0,2)$}\\ \ground & \text{if } (g,n)=(0,2)\end{cases}$. Here $\S_2$ acts trivially
on $\O_+((0,2))$. Furthermore, define the structure of an extended modular $\D$-operad on $\O_+$ by the following composition
\[\overline{\M}_\D\O_+((g,n))=\colim_{G\in\Iso \Ga((g,n))}\D(G)\otimes\O_+((G))\ra \colim_{G\in\Iso \Gamma((g,n))}\D(G)\otimes\O_+((G))\ra\O_+((g,n))\]
where the first arrow is induced by the functor $\Ga((g,n))\ra\Gamma((g,n)): G\mapsto \tilde{G}$ and the second one is given by the structure
of a modular $\D$-operad on $\O$. The canonical basis
element $\mathbf{1}$ of $\O_+((0,2))$ is called the \emph{unit} of $\O_+$, and the extended modular $\D$-operad $\O_+$ is said to be obtained from $\O$ by adjoining a unit.
\end{defi}

From now on we will assume that $\D$ is a \emph{cocycle}, i.e. $\D(G)$ is
one-dimensional (concentrated in either even or odd degree)
for any stable graph $G$ and that the corresponding structure maps $\nu_f$ are isomorphisms. In that case an $\overline{\M}_\D$-algebra structure on $\O$ gives rise
to (possibly parity reversing) operadic composition maps $\circ_{i}:\O((m))\otimes\O((n))\ra \O((m+n-2))$ and $\xi_{ij}:\O((n))\ra\O((n-2))$
by restricting the structure map $\overline{\M}_\D\O\ra\O$ to the graphs having only one  edge.

We are now ready to define our main object of study. Recall from \cite{GeK} that the \emph{Feynman transform} $\FDo$ of a modular
$\D$-operad $\O$ is a modular $\mathfrak{K}\otimes\D^*$-operad whose underlying
stable
$\S$-module is
$$\M_{\mathfrak{K}\otimes\D^*}(\O^*)((g,n))=\colim_{G\in\Gamma((g,n))}\mathfrak{K}(G)\otimes\D(G)^*\otimes
\O((g,n))^*.$$
 The differential on $\FDo$ is the sum of the `internal' differential induced by that on $\O$ and the differential induced by the edge-expansions of stable graphs and the operadic composition in $\O$.
\begin{defi}
Let $\D$ be a cocycle,  and $\O$ be a modular $\D$-operad. The \emph{dual Feynman
transform} of $\O$  is the quotient of the extended modular $\D$-operad $(\O[s])_+$, where $s$ has degree~$1$, by the relation $s^2:=s\circ_1 s=0$.
It will be denoted by $\FvDo$. The differential~$d$
on $\FvDo$ is the unique extension of the differential on~$\O$ satisfying $d(s)=\mathbf{1}$.
\end{defi}

\begin{theorem}\
\begin{enumerate}
\item
The dg vector space $\FvDo((g,n))$ is contractible for $n>0$ and any $g$.
\item
There is an isomorphism of dg vector spaces
$\FvDo((g,0))^*\cong \FDo((g,0))$.
\end{enumerate}
\end{theorem}\label{main}
\begin{proof}
According to Lemma \ref{extended} we have
\[\O[s]((g,n))\cong \colim_{G\in \Iso\Ga((g,n))}\mathfrak{K}(G)\otimes\D(\tilde{G})\otimes\O((\tilde{G})).\] Recall that the
edges of a decorated extended stable graph $G$ represent the operation $s$, and unstable bivalent vertices between  edges correspond to composing $s$ with itself. Since these compositions are zero in $\FvDo$ we conclude that
the only nonzero contributions to the direct sum are from graphs all of whose unstable vertices are adjacent to legs.

Let $n>0$ and $(g,n)\neq (0,2)$. For any stable graph $G\in \Gamma((g,n))$ we need to consider the $2^n$ graphs obtained from $G$ by placing
unstable bivalent vertices on a subset (possibly empty) of legs of $G$. Denote the space spanned by such graphs by $V$; we see therefore that the $(g,n)$-component of $\FvDo$ is isomorphic to  \[\bigoplus_{G\in [\Gamma((g,n))]}\left(\mathfrak{K}(G)\otimes\D(G)\otimes\O((G))\right)_{\Aut(G)}\otimes V.\]
It is clear that $V$ is isomorphic to the tensor product of contractible dg vector spaces of dimension two: $\ground\cdot \mathbf 1\oplus \ground\cdot s$ with $d(s)=\mathbf 1$.
In particular $\FvDo((g,n))$ is contractible. The contractibility of $\FvDo((0,2))$ is obvious since the latter space is spanned by $\mathbf{1}$ and $s$ and $d(s)=\mathbf{1}$. Part (1) of our theorem is proved.

Now let $n=0$. We saw that the unstable vertices are allowed on the legs only and thus we are reduced to considering
the usual \emph{stable} graphs.  We have thus:
\[\O[s]/s^2((g,0))\cong \colim_{G\in \Iso\Gamma((g,0))}\mathfrak{K}(G)\otimes\D({G})\otimes\O(({G})).\]
Recalling our assumption that the characteristic of $\ground$ is zero,
this establishes an isomorphism \[\FvDo((g,0))^*\cong\O[s]/s^2((g,0))^*\cong \FDo((g,0))\] as $\Z/2$-graded vector spaces. The differential on
$\O[s]/s^2((g,0))$ is the sum of the internal differential on tensor powers of $\O$ and the differential induced by the equation $d(s)=\mathbf{1}$. The latter
is the usual graph differential corresponding to the contractions of edges of stable graphs and given by operadic compositions in $\O$. The total differential is therefore dual to that
in $\FDo((g,0))$.
\end{proof}
\begin{rem}
Let us represent the identity operation by the graph $\begin{xymatrix} @C=2ex@R=2ex@M=-.1EX{
\ar@{-}[r]& \bullet  \ar@{-}[r]&}\end{xymatrix}$ and the operation $s$ by the graph $\begin{xymatrix}  @C=2ex@R=2ex@M=-.1EX{
\ar@{-}[r]& \bullet \ar@{-}[rr] && \bullet \ar@{-}[r]&}\end{xymatrix}$ . An element in $\FvDo$ will be represented
by a graph whose stable vertices are decorated by elements in $\O$ inscribed in circles and whose unstable vertices (which are only allowed on the legs) are represented as dots. The differential is induced by contractions of  edges plus the internal differential in $\FvDo$. For example the equations $d(s)=\mathbf{1}$ and $d(\mathbf{1})=0$ are represented as
\[d(\begin{xymatrix}  @C=2ex@R=2ex@M=-.1EX{
\ar@{-}[r]& \bullet \ar@{-}[rr] && \bullet \ar@{-}[r]&}\end{xymatrix})=\begin{xymatrix} @C=2ex@R=2ex@M=-.1EX{
\ar@{-}[r]& \bullet  \ar@{-}[r]&}\end{xymatrix} ~\text{~and~}~ ~d(\begin{xymatrix} @C=2ex@R=2ex@M=-.1EX{
\ar@{-}[r]& \bullet  \ar@{-}[r]&}\end{xymatrix})=0\]

 The operadic composition in $\FvDo$ is described by the following pictures.

\[ \xymatrix{&&&&&&&&&&\\&*+[o][F]{\scriptscriptstyle{x}}*=0\ar@{-}[l]\ar@{-}[dl]\ar@{-}[ul]\ar@{-}[r]&\circ&*+[o][F]
{\scriptscriptstyle{y}}*=0\ar@{-}[l]\ar@{-}[ur]\ar@{-}[dr]&\ar@{-}[l]&\ar@{=}[r]&
&&*+[o][F]{\scriptscriptstyle{x\circ y}}*=0\ar@{-}[l]\ar@{-}[dl]\ar@{-}[ul]\ar@{-}[ur]\ar@{-}[dr]&\ar@{-}[l]\\&&&&&&&&&& }
\]
\[ \xymatrix{&&&&&&&&&&&&\\&*+[o][F]{\scriptscriptstyle{x}}*=0\ar@{-}[l]\ar@{-}[dl]\ar@{-}[ul]\ar@{-}[r]&*=0{\bullet}\ar@{-}[r]&\circ&*+[o][F]
{\scriptscriptstyle{y}}*=0\ar@{-}[l]\ar@{-}[ur]\ar@{-}[dr]&\ar@{-}[l]&\ar@{=}[r]&
&&*+[o][F]{\scriptscriptstyle{x}}*=0\ar@{-}[l]\ar@{-}[dl]\ar@{-}[ul]\ar@{-}[r]&*+[o][F]{\scriptscriptstyle{y}}*=0\ar@{-}[ur]\ar@{-}[dr]&\ar@{-}[l]\\&&&&&&&&&&&&& }
\]
 Thus, the composition in  the $\FvDo$ is not the usual glueing of graphs (which, after all, would not be compatible with our contracting differential), but the glueing \emph{followed  by the contraction of the newly formed edge}. Self-glueings are determined similarly, taking into account the usual convention that contracting a loop results in the genus of the corresponding vertex increasing by one. According to our definition of adjoining a unit to a modular operad the glueing of the two edges of $\begin{xymatrix} @C=2ex@R=2ex@M=-.1EX{
\ar@{-}[r]& \bullet  \ar@{-}[r]&}\end{xymatrix}$ is equal to zero.
The relation $s^2=0$ translates into the following picture.
\[ \xymatrix{&&&&&&&&&&&&\\&*+[o][F]{\scriptscriptstyle{x}}*=0\ar@{-}[l]\ar@{-}[dl]\ar@{-}[ul]\ar@{-}[r]&*=0{\bullet}\ar@{-}[r]&\circ&
*=0{\bullet}\ar@{-}[l]&*+[o][F]
{\scriptscriptstyle{y}}*=0\ar@{-}[l]\ar@{-}[ur]\ar@{-}[dr]&\ar@{-}[l]&\ar@{=}[r]&&\text{zero}\\&&&&&&&&&&&&& }
\]
The last equation gives a graphical explanation why the unstable
vertices -- the dots -- do not appear on the  edges.
\end{rem}

\section{Twisted modular operads and their algebras}\label{twistedalgebras}
In this section we describe algebras over  modular $\D$-operads and discuss some important examples. We will only be concerned with the case $\D=\De$ where $d$ is an integer modulo $2$, i.e. $\D(G)=\Det(H_1(G))^{\otimes d}$ for a stable graph $G$. In particular, if $d=0$ this corresponds to an (untwisted) modular operad. Note that since for a tree graph $G$ the twisting by $\De$ is trivial the genus zero part of any modular $\De$-operad is a usual cyclic operad.
Let $V$ be a dg vector space with a symmetric inner product of degree $d$, i.e. an (even) symmetric map $V\otimes V\ra\Pi^d\ground$. Then the stable $\S$-module $\{\mathcal E((g,n))=V^{\otimes n}\}$
forms a modular $\mathfrak{K}^{d}$-operad, see \cite{bar}. Note that the genus zero part of $\mathcal E$ is not a cyclic operad unless $d=0$; to remedy the situation we twist the operad $\mathcal E$ by the $\S$-module ${\mathfrak s}^{d}\otimes\Pi^{d}$, i.e. consider instead the operad \[\tilde{\mathcal E}((g,n)):={\mathfrak s}^{d}((g,n))\otimes\Pi^{d}((g,n))\otimes\mathcal E((g,n))
\cong\Pi^d(\Pi^d V)^{\otimes n}
.\]
Then $\tilde{\mathcal E}$ is a modular $\De$-operad, and its genus zero part $\tilde{\mathcal E}(0,\bullet)$ forms a cyclic operad. Using the isomorphism $V\cong \Pi^d(V^*)$ we see that that $\tilde{\mathcal E}(0,\bullet)$ is isomorphic to the endomorphism cyclic operad $\{\Hom(V^{\otimes(n-1)},V)\}$.
\begin{defi}
Let $\Cyc$ be a cyclic operad
and $V$ be a dg vector space with equipped with an inner product of degree $d$.
Then a (cyclic) $\Cyc$-algebra structure on $V$ is
 a homomorphism of cyclic operads $\Cyc\ra\tilde{\mathcal E}(0,\bullet)$.
\end{defi}
The notion of a
$\Cyc$-algebra is easy to work out explicitly in the case of a quadratic cyclic operad $\Cyc$. In that case $\Cyc$ is generated by
a set of cyclically invariant elements in $\Cyc((3))$ and a
$\Cyc$-algebra structure on $V$ is specified by a collection of binary operations on $V$ which are cyclically invariant when regarded as elements in $V^{\otimes 3}$.
\begin{example}Let us consider what this definition specializes to in the case of the three standard cyclic operads: commutative, associative and Lie.
\begin{enumerate}\item
Let $\Cyc=\Comm$ be the commutative non-unital cyclic operad: $\Cyc((n))=\begin{cases}0, & n=1,2\\ \ground,&  n\geq 3 \end{cases}$. The operad $\Cyc$ is generated by a single element of even degree in  $\Cyc((3))$ with the trivial action of $\S_3$. It follows that a
cyclic $\Comm$-algebra is a (graded)-commutative dg algebra $A$ with an inner product $\langle,\rangle$ of degree $d$ which is invariant in the sense that $\langle ab,c\rangle=\langle a,bc\rangle$ for any $a, b, c\in A$. In other words, $A$ is a dg commutative Frobenius algebra.\item
Let $\Cyc=\Ass$ be the associative non-unital cyclic operad: $\Cyc((n))=\begin{cases}0, & n=1,2\\ \ground[\S_n/\mathbb{Z}_n], &  n\geq 3 \end{cases}$. The operad $\Cyc$ is generated by two elements of even degree in  $\Cyc((3))$. It follows that an
cyclic $\Ass$-algebra is an associative dga $A$ with an inner product $\langle,\rangle$ of degree $d$ which is invariant in the sense that $\langle ab,c\rangle=\langle a,bc\rangle$ for any $a, b, c\in A$. In other words, $A$ is a (possibly noncommutative) dg Frobenius algebra.\item
Let $\Cyc=\Lie$ be the Lie cyclic operad. It is generated by a single element of even degree in  $\Cyc((3))$ on which $\S_3$ acts by the sign representation; the relations come from the Jacobi identity. It follows that a
cyclic $\Lie$-algebra is a dg Lie algebra $A$ with an inner product $\langle,\rangle$ of degree $d$ which is invariant in the sense that $\langle [a,b],c\rangle=\langle a,[b,c]\rangle$ for any $a, b, c\in A$. For example, any reductive Lie algebra is a cyclic $\Lie$-algebra. The commutator Lie algebra  of an
$\Ass$-algebra is a
$\Lie$-algebra.
\end{enumerate}
\end{example}
Now let $\O$ be a modular $\De$-operad and $\tilde{\mathcal E}=\{\tilde{\mathcal E}((g,n))\}$ be the $\De$-operad associated to a dg vector space $V$ with a symmetric inner product of degree $d$ as defined above.
\begin{defi}\label{modularalgebra}
The structure of a (modular) $\O$-algebra on $V$ is a homomorphism of $\De$-modular operads $\O\ra\tilde{\mathcal{E}}$.
\end{defi}
Recall that for any stable graph $G\in \Gamma((g,n))$ a modular $\De$-operad $\O$ determines a homomorphism $\mu_G:\De(G)\otimes\O((G))\ra\O((g,n))$ which corresponds to taking operadic compositions in $\O((G))$ along the internal edges of $G$. For a structure of an $\O$-algebra on $V$ given by a map $f:\O\ra \tilde{\mathcal{E}}$ we can form the composition \[f\circ\mu_G:\De(G)\otimes\O((G)) \ra\tilde{\mathcal{E}}((g,n))\cong \Pi^{d(n-1)}V^{\otimes n}.\]
We will denote $f\circ\mu_G$ by
$Z_V(G)$
 and call it the \emph{Feynman amplitude} of the decorated graph $G$ corresponding to the algebra $V$.

We are going to introduce certain modular operads closely related to $\Ass$, $\Comm$ and $\Lie$, and describe the algebras over them. To do this we need to use the \emph{modular closure} of a cyclic operad  introduced in \cite{HW}. We will use a slightly more general notion than that considered in the cited reference in order to include the case of $\Det$-operads.

The functor of restricting to the genus zero part of a modular $\De$-operad  takes values in the category of cyclic operads. This functor admits a left adjoint; the value of the latter on a cyclic (non-unital) operad $\Cyc$ will be denoted by $\Cycod$. For $d=0$ we will shorten this notation to $\Cyco$. The existence of $\Cycod$ follows immediately from the description of twisted modular and cyclic operads as algebras over appropriate triples. Indeed, let $\Cyc$ be generated by a cyclic $\S$-module $\{\Cycgen((n))\}$. We can consider it as a stable $\S$-module with $\Cycgen((g,n))=0$ for $g\neq 0$. Then $\Cycod$ is constructed from the free $\De$-operad $\M_{\De}\Cyc$ by quotienting out by the same relations as those defining~$\Cyc$. We will  call~$\Cycod$ the $\De$-modular closure (or simply modular closure for $d=0$) of~$\Cyc$.
We have the following result which follows immediately from definitions.
\begin{prop}
There is a one-to-one correspondence between cyclic $\Cyc$-algebra structures on a dg vector space $V$ with an inner product
of degree $d$ and modular $\Cycod$-algebra structures on $V$.
\end{prop}\noproof
\begin{rem}
There is another functor from cyclic operads to modular $\De$-operads which is \emph{right} adjoint to the genus zero part functor.
We will denote the value of this functor on a cyclic operad $\Cyc$ by $\Cycud$, and will shorten this to $\Cycu$ for $d=0$. Algebras over $\Cycud$ are characterized by the property that their Feynman amplitudes on stable graphs of positive genus are always zero. Note that $\Cycud$ is the quotient of $\Cycod$ by the ideal generated by the images of all contraction maps $\xi_{ij}:\Pi^d\Cycod((I))\ra \Cycod((I\setminus \{i,j\}))$. One can consider other modular operads, intermediate between $\Cycod$ and $\Cycud$ which have the property that some, but not all, contraction maps are trivial.
\end{rem}
We will now show that the modular $\De$-operads naturally give rise to two-dimensional surfaces. Denote by $M_g(n)$ the 
category of compact oriented $2$-dimensional surfaces  of genus $g$ with $n$ parametrized and labeled boundary components; morphisms are homeomorphisms preserving orientation, parametrizations and labelings.

Consider the modular operad $\TFT^d$ such that
$${\TFT}^d((g,n))=\begin{cases}\ground^{[M_g(n)]} &\text{if $d=0$, or $d=1$ and $g=0$,}\\ $0$ &\text{if $d=1$ and $g>0$.}\end{cases}$$

 The operadic compositions are induced by the glueing or self-glueing of surfaces at the boundaries. In case $d=1$ a surface of positive genus obtained as a result of glueing is regarded to be zero. This clearly makes $\TFT^d$ into a modular $\De$-operad.  Algebras over $\TFT^d$ are called (twisted) topological field theories.

Similarly define the modular $\De$-operad $\OTFT^d$ as follows. Consider a two-dimensional oriented compact surface  of genus $\gamma$ with $\nu$ boundary components, equipped with non-overlapping orientation-preserving embeddings of
$n$ parametrized and labeled intervals into its boundary. The
category of such surfaces, with morphisms being homeomorphisms preserving orientation, parametrizations and labelings, will be denoted by
$M_{\gamma,\nu}(n)$.
Set
$${\OTFT}^d((\gamma,\nu,n))=\colim_{G\in M_{\gamma,\nu}(n)}\De(H_1(G))$$
and
$${\OTFT}^d((g,n))=\bigoplus_{2(\gamma-1)+\nu=g-1,\enspace \nu>0}{\OTFT}^d((\gamma,\nu,n)).$$


 The operadic compositions on $\OTFT$ are induced by the glueing  or self-glueing of surfaces at the
 \emph{open boundaries}, i.e. at the embedded intervals. This makes $\OTFT^d$ into a modular $\De$-operad.  Algebras over $\TFT^d$ are called (twisted) open topological field theories.

\begin{rem}\label{OTFTorientation}
Let $G$ be a oriented surface with boundary. Then an ordering of the boundary components of $G$ specifies a canonical nonzero element of $\Det(H_1(G))$. To see this, consider the following exact Mayer-Vietoris sequence:
$$0\ra H_2(G_c)\ra H_1(\partial G) \ra H_1(G) \ra H_1(G_c)\ra 0,$$
where $G_c$ is the closed surface obtained from $G$ be glueing a disk into each boundary component. Taking the canonical generators of $\Det(H_2(G_c))=\Pi H_2(G_c)$ and of $\Det(H_1(G_c))$, we obtain an isomorphism
$\Det(H_1(G))\cong\Det(H_1(\partial G))$.

It follows that for
$G\in M_{\gamma,\nu}(n)$, the action of $\Aut(G)$ on  $\Det(H_1(G))$ is trivial if and only if the surface $G$ has at most one \emph{free boundary component}, i.e. boundary component not containing any open boundaries. Hence $\OTFT^1((\gamma,\nu,n))\cong\bigoplus\De(H_1(G))$ where the sum is over isomorphism classes of such surfaces.
\end{rem}

\begin{theorem}\label{tft}
There are isomorphisms of modular $\De$-operads:\begin{enumerate}
\item
$\Commod\cong \TFT^d$
\item
$\Assod\cong \OTFT^d.$
\end{enumerate}
\end{theorem}
\begin{proof}
To prove (1) let $C$ be the underlying stable $\S$-module for the (non-unital) operad $\Comm$. In other words,
$C((g,n))=\begin{cases}\ground & \text{for $g=0, n\geq 3$}\\ $0$ & \text{for $g\neq 0$ or $n=0,1,2$}\end{cases}$.
The free modular $\De$-operad
$\M_{\De}C$
generated by $C$ is spanned by
isomorphism classes of $\De$-graphs with vertices of valence at least three.
The modular $\De$-operad $\Commod$ is the quotient of $\M_{\De}C$ by the associativity relation. This means that two graphs are considered
equivalent if one is obtained from the other by a sequence of expansions or contractions of edges that are not loops. To any such graph we associate a two-dimensional surface obtained by replacing each vertex by a sphere $S^2$ and each edge by a thin tube homeomorphic to the cylinder $[0,1]\times S^1$.
Contracting an edge clearly results in a homeomorphic surface and we thus obtain a map $f:\Commod\ra \TFT^d$ of modular $\De$-operads. Let $d=0$. The space $\TFT^d((g,n))$ is spanned by the elements corresponding to spheres with $g$ handles and $n$ closed boundary components. This corresponds to the graph with one vertex, $g$ loops and $n$ legs. Thus, the map $f$ is surjective. To see that it is injective let $G_1$ and $G_2$ be two graphs which give rise to homeomorphic surfaces. Let $g$ and $n$ be their genus and the number of boundary components respectively. Since $G_1$ and $G_2$ are equivalent to graphs having only one vertex it follows that these graphs should both have $g$ loops and $n$ legs and therefore they are isomorphic. This shows the injectivity of $f$. The case of an odd $d$ is even simpler since a graph $G$ having a loop admits an automorphism acting as multiplication by $-1$ on $\Det(H_1(G))$; thus such graphs do not contribute to $\Commod$. Similarly surfaces of positive genus do not contribute to $\TFT^d$ and so we are left with just spheres with boundaries. It follows that in this case $\Commod\cong\Commud\cong \TFT^d$.

The proof of (2) is similar. The free $\De$-modular operad generated by the stable $\S$-module underlying $\Ass$ is the space spanned by isomorphism classes of $\De$-twisted \emph{ribbon} graphs; the associativity relation allows one to contract or expand any edge that is not a loop. The ribbon structure allows one to replace each edge and leg by a thin strip which results in a surface with boundary. Note that the legs turn into parametrized intervals embedded into the boundary. As in the commutative case we obtain a map  $\Assod\ra \OTFT^d$ of modular $\De$-operads. The surjectivity of the map follows from the classification of topological surfaces with boundary. To see that it is injective it suffices to show that any homeomorphism of surfaces could be realized as a sequence of edge-contractions and expansions. It suffices to deal with the case $n=0$ when the graphs have no legs. Consider the category whose objects are ribbon graphs of genus $\gamma$ and with $\nu$ boundary components and whose morphisms are generated by isomorphisms and edge-contractions. It is known, cf. for example \cite{Igusa} that the classifying space of this category is homeomorphic to the moduli space of Riemann surfaces of genus $\gamma$ with $\nu$ boundary components. The desired statement then follows from the well-known fact that this moduli space is connected.
\end{proof}
\begin{rem}
Theorem \ref{tft} is a variation on the well-known theorem due to Atiyah et al. (see, e.g. \cite{Kock}) which states that a tensor functor into the category of vector spaces from the category of closed $1$-dimensional manifolds and cobordisms between them is equivalent to a Frobenius algebra. Our theorem is different in several respects. First of all, we treat algebras over modular operads rather than functors from the corresponding monoidal categories. Second, we also discuss \emph{open} topological field theories and show that they give rise to not necessarily commutative Frobenius algebras.
Finally, we include also the twisted versions of the corresponding theories which give rise to Frobenius algebras with an odd inner product. In this connection we mention the result of Costello \cite{Cos} that
open topological conformal field theories are in one-to-one correspondence with Calabi-Yau $A_\infty$-categories which could be viewed as a derived and categorified version of part (2) of Theorem \ref{tft}.
\end{rem}

\begin{rem}
As mentioned in the proof of part (1) of Theorem \ref{tft}, for $d=1$ there is an isomorphism $\Commod\cong\Commud$.
Therefore the Feynman amplitude of a commutative Frobenius algebra with an odd inner product on a stable graph of positive genus is zero. In the case $d=0$ the vanishing of such amplitudes is a rather strong condition on a Frobenius algebra.
\end{rem}
Next we introduce a modular $\De$-operad lying between $\Assud$ and $\Assod$.
\begin{defi}
The modular $\De$-operad $\KAssd$ is the quotient of $\Assod$ by the ideal generated by all surfaces with at least one free boundary component.
\end{defi}
We will see later on that the operad $\KAssd$ appears in the construction of a certain compactification of decorated Riemann surfaces due to Kontsevich who used it in his proof of the Witten conjecture \cite{Kon}.

\begin{rem}\label{crapsupercrap}Representing surfaces as ribbon graphs we see that $\KAssd$ is obtained from $\Assod$ by imposing the relation
\begin{equation}\label{rel1}\xymatrix{&*=0{\bullet}\ar@{-}[r]\ar@{-}@(ul, dl)\ar@{-}@(dl,ul)[]& }=0.\end{equation}
Similarly $\Assud$ is obtained from $\Assod$ by imposing both relation~(\ref{rel1}) and the further relation
\begin{equation}\label{rel2}\xymatrix{&*=0{\bullet}\ar@{-}[l]\ar@{-}[r]\ar@{-}@(ul, dl)\ar@{-}@(dl,ul)[]& }=0.\end{equation}
The left hand side of (\ref{rel2}) is the `double twist' operator appearing in the Cardy condition in Moore and Segal's formulation of open-closed topological field theory \cite{MooreSegal}, \cite{laz}, \cite{nata}.
These relations translate into certain identities which should hold for the algebras over $\KAssd$ and~$\Assud$. For example, let $A$ be an algebra over $\Assod$ and $\sum_i a_i\otimes b_i\in A\otimes A$ be the inverse to the invariant form on $A$. Then equation (\ref{rel1}) implies that in order that $A$ be an algebra over $\KAssd$ we should have $\sum_ia_ib_i=0$. Similarly we see from equation (\ref{rel2}) that $A$ further lifts to an algebra over $\Assud$ if an only if the following additional condition holds: for any $x\in A$ we have $\sum_i(-1)^{|a_i||x|}a_ixb_i=0$.

\end{rem}

\begin{rem}\label{operadzoo}  While $\KAssd$ appears to be the most important quotient of $\Assod$, we list some others that might be of interest:
\begin{itemize}
\item $\RAssd$: the quotient of $\Assod$ by the ideal generated by all surfaces of positive genus, equivalently modulo the relation
\begin{equation}\label{rel3}\xymatrix{&*=0{\bullet}\ar@{-}[r]\ar@{-}@(u, l)\ar@{-}@(ul,dl)[]& }=0.\end{equation}
\item $\KRAssd$: the quotient of $\Assod$ by the ideal generated by all surfaces with at least one free boundary component and
 all surfaces of positive genus, i.e. modulo relations (\ref{rel1})
 and (\ref{rel3}).
 \item $\SAssd$: the quotient of $\Assod$ by the relation~ (\ref{rel2}).
 \end{itemize}
 We have the following commutative diagram of modular $\De$-operads; each arrow is a quotient map:
 \[\xymatrix{&\Assod\ar[dl]\ar[dr]\\ \KAssd\ar[dr]&&
 \RAssd\ar[dl]\ar[dr]\\&
 \KRAssd\ar[dr]&&
 \SAssd\ar[dl] \\ &&\Assud}\]
 We remark that $\SAssd$ is rather close to $\Assud$; indeed we have
$\SAssd((g,0))\cong\Assud((g,0))$ for all $g>0$.

The surface corresponding to the ribbon graph appearing in relation (\ref{rel3}) is a torus with one boundary component containing one open boundary. The relation
 translates, in the notation of Remark~\ref{crapsupercrap}, to the condition $\sum_{i,j}(-1)^{|a_j||b_i|}
 a_ia_jb_ib_j=0$ for an $\Assod$-algebra to lift to a $\RAssd$-algebra.  It is unclear whether this equation has any significance from an algebraic point of view.
\end{rem}

\section{Algebras over the dual Feynman transform
}
In this short section we give a general formulation of Kontsevich's dual construction for algebras over a modular $\De$-operad.
Before stating the result, we extend Definition~\ref{modularalgebra} in an obvious way.
Let $\emo$ be an extended modular $\De$-operad. Given a dg vector space $V$ equipped with an inner product of degree $d$, the operad ${\tilde{\mathcal{E}}}$ constructed in Section~\ref{twistedalgebras} may be regarded as an extended
modular $\De$-operad.
An $\emo$-algebra structure on $V$ is defined to be a homomorphism
$\emo\ra\tilde{\mathcal{E}}$ of extended modular $\De$-operads. If $\emo$ is equipped with a unit $\mathbf{1}\in\emo((0,2))$, we require $\mathbf{1}$ to be mapped to the identity operator in
$\tilde{\mathcal{E}}((0,2))\cong\Hom(V,V)$.

\begin{theorem}\label{dualalg}
Let $\O$ be a modular $\De$-operad and $V$ be dg vector space with an inner product $\langle,\rangle$ of degree $d$. Then the structure of an $\FvDeo$-algebra on $V$ is equivalent to the following data:
\begin{enumerate}\item
The structure of an $\O$-algebra on $V$.\item
An odd operator $s:V\ra V$ such that
\begin{itemize}\item
$s^2=0$;\item
$(ds+sd)(a)=a$ for any $a\in V$;\item
$\langle s(a),b\rangle=(-1)^{|a|}\langle a,s(b)\rangle$.
\end{itemize}
\end{enumerate}
\end{theorem}
In other words, an $\FvDeo$-algebra is an $\O$-algebra together with a contracting homotopy having square zero and compatible with the given inner product.
\begin{proof}
By definition, the extended modular $\De$-operad $\FvDeo$ is generated by $\O$ and an odd element $s\in \FvDeo((0,2))$ of square zero; this translates into the existence of an odd square-zero operator on $V$ which we denote by the same symbol $s$. The compatibility condition between $s$ and $d$ is equivalent to the identity $d(s)=\mathbf{1}$ in $\FvDeo$. Finally, since the action of $\Z_2$ is trivial on $s\in \FvDeo$ the corresponding operator $s\in\Hom(V,V)$ is likewise $\Z_2$-invariant which is equivalent to the stated compatibility condition between $s$ and the inner product $\langle,\rangle$.
\end{proof}
In particular, taking Feynman amplitudes of $V$ on $\De$-twisted $\O$-decorated graphs without legs (the partition function of $V$)  we obtain a chain map
\[\FvDeo((\bullet,0))\cong \FDeo((\bullet,0))^*\longrightarrow\Pi^d\ground\]
or, in other words, a cycle in the dg vector space $\Pi^d\FDeo((\bullet,0))$.

\begin{rem}

Let $V$ be an $\FvDeo$-algebra, and let $s:V\ra V$ be the contracting homotopy given by Theorem~\ref{dualalg}. Then it is easy to see that $\im s$ is an isotropic complement to the isotropic subspace $\Ker d=\im d$ in $V$.

Conversely let $V$ be a contractible $\O$-algebra, and let $U'$ be an arbitrary isotropic complement to $U:=\Ker d=\im d$ in $V$. Then $d$ restricts to an isomorphism of $U'$ onto $U$, and we define $s:V\ra V$ to be inverse to $d$ on $U$ and zero on $U'$. It is straightforward to check that $s$ satisfies the required conditions of Theorem~\ref{dualalg}, making $V$ an $\FvDeo$-algebra.

From this characterization it follows that \emph{any} contractible $\O$-algebra can be given the structure of an $\FvDeo$-algebra.

\end{rem}



\section{Stable graph complexes} \label{stablegraphs}

We have seen that the dual Feynman transform of a modular $\D$-operad $\O$ may be viewed as a space of twisted $\O$-decorated stable graphs. In this section we focus our attention on the modular
$\De$-operads directly related to $\Comm$ and $\Ass$ that were defined in the previous section, namely $\O=\Commu,\Commo,\Assu,\KAss,\Asso$ and their twisted versions. In these cases we shall use Theorem \ref{tft} to interpret twisted $\O$-decorated stable graphs in purely combinatorial terms, as `oriented' graphs with additional combinatorial structure on vertices.

Thus we define certain \emph{graph complexes} associated to these modular $\De$-operads, which serve as explicit models for their (dual) Feynman transforms. The associated cell complexes were introduced and studied in \cite{Kon}, \cite{Loo}, \cite{Zv} and \cite{Mon}; however our explicit description of the corresponding chain complexes is new.


Recall that we are considering (twisted) modular operads in dg vector spaces. Their Feynman transforms are therefore also (twisted) modular operads in dg vector spaces. Suppose however that $\O$ is a modular $\De$-operad in ($\Z/2$-graded) vector spaces, i.e. with vanishing differential. Then $\FDo$ is a cochain complex with respect to the following $\Z$-grading:
\begin{defi}
Let $\O$ be a modular $\De$-operad with vanishing differential.  Then for $i=0,1,2,\ldots$ the $i$th  grading component of the Feynman transform of $\O$ is
\[\F^i_{\De}\O:=\bigoplus_{G}\left[\mathfrak{K}(G)\otimes\De(G)\otimes\O((G))^*\right]_{\Aut G}\] where the direct sum is extended over the isomorphism classes of stable graphs with $i$ edges.
\end{defi}
\begin{rem}
It is clear that $\FDeo$ actually has a $\Z\times\Z/2$-grading and the $\Z$-grading defined above is obtained by simply forgetting the $\Z/2$-grading.
\end{rem}

\subsection{Commutative case} We start with the simplest case of commutative graphs. A \emph{stable commutative graph} is essentially what Getzler and Kapranov called a \emph{stable graph} except we will not consider graphs with legs. Furthermore, whenever a vertex has genus $g$ we will attach to it $g$ dotted loops. In other words, we represent a commutative stable graph as a usual graph having no vertices of valence $0$, $1$ or $2$, such that some of its loops are dotted as in the following picture.\\ \\
\[\xymatrix{&*=0{\bullet}\ar@{-}\ar@{-}[r]\@{-}@(ul, dl)\ar@{-}@(dl,ul)[]\ar@{-}@/_-2pc/[rr]&*=0{\bullet}\ar@{-}[l]
\ar@{-}[r]\@{-}@(ur, dl)\ar@{-}@(ur,ul)[]\ar@{.}@(dr,dl)[]&*=0{\bullet}\ar@{-}@/^2pc/[ll]\ar@{.}@(dr,ur)[]& }\] \\ \\
Whenever we say simply `edge' or `loop' without an adjective we shall mean `solid edge' and `solid loop'. The set $\Edge(G)$ will refer to the set of edges but $H_1(G)$ will refer to the first homology group of $G$ \emph{taking into account the dotted edges}.

We define a $0$-orientation on a stable commutative graph $G$ to be an ordering of the edges of $G$ modulo even permutations. Similarly we define a $1$-orientation to be an ordering of the vertices together with an orientation of \emph{every} edge of $G$, again modulo even permutations.
A $d$-orientation on $G$ determines a non-zero element
of $\Det(\Edge(G))\otimes\Det^d(H_1(G))$, see, e.g. \cite{CV}; reversing the orientation means negating the corresponding element.


We will now introduce chain complexes $\G_\bullet^{\Commod}$ for $d=0,1$. The underlying vector space for $\G_\bullet^{\Commod}$ is spanned by the isomorphism classes of $d$-oriented stable commutative graphs, modulo the relation that oppositely oriented graphs sum to zero.

The differential is defined as follows. Let $G$ be a stable commutative graph with a $d$-orientation and $e$ be an edge of $G$. If $e$ is not a loop we denote by $G_e$ the graph resulting from the contraction of $e$.  If $e$ is a loop then $G_e$ is obtained from $G$ by replacing $e$ with a dotted loop. In any case $G_e$ inherits a $d$-orientation from $G$ and we define the differential by
\begin{equation}\label{dif}d(G)=\sum_{e\in\Edge(G)}G_e.\end{equation}
We denote by $\G_\bullet^{\Commod}(g)$ the subcomplex of
$\G_\bullet^{\Commod}$
 spanned by $d$-oriented stable graphs of genus $g$; it is clear that these subcomplexes form a direct sum. The complex $\G_\bullet^{\Commod}(g)$ will be called the \emph{stable graph complex} for $d=0$ and \emph{twisted stable graph complex} for $d=1$.
The corresponding homology will be called \emph{stable graph homology} and \emph{stable twisted graph homology} respectively; it will be denoted by $H_\bullet^{\Commod}(g)$.
\begin{rem}
Our terminology `graph complexes' and `twisted graph complexes' is in agreement with that of Getzler and Kapranov since these notions correspond to the dual Feynman transform of a modular operad and that of a twisted modular operad respectively; see Proposition \ref{stablecom} below.
However since the Feynman transform of a modular $\De$-operad is (after an appropriate transfer of structure) a modular $\Det^{1-d}$-operad, it might also be reasonable to interchange the labels `twisted' and `untwisted' -- this alternative convention is adopted in \cite{kontfeynman} and \cite{kontsg}. The same remark applies to the various versions of ribbon graph complexes defined below.
\end{rem}

The spaces spanned by graphs having at least one dotted loop form subcomplexes in $\G_\bullet^{\Commod}(g)$; quotienting out by these subcomplexes we obtain the complex of (commutative) graphs $\G_\bullet^{\Commud}(g)$.  This corresponds to the fact that the (twisted) modular  operad $\Commud$ is a quotient of $\Commod$.

It will be convenient for us to use the corresponding cochain complexes. For $\O=\underline{\Comm}^d$ or $\O=\Commod$ we denote by $\G^\bullet_{\O}(g)$ the cochain complex which is $\ground$-dual to $\G_\bullet^{\O}(g)$ so that
$\G^i_{\O}(g)=\Hom(\G_i^{\O}(g),\ground)$. The corresponding cohomology will be denoted by $H^i_{\O}(g)$.

The following proposition compares the stable commutative graph complexes with the Feynman transforms of the corresponding $\De$-modular operads
\begin{prop}\label{stablecom}
For $\O=\underline{\Comm}^d$ or $\O=\Commod$ there is an isomorphism of complexes
\[\F^\bullet_{\De}\O((g,0))\cong\G^\bullet_{\O}(g).\]
\end{prop}
\begin{proof}
Direct inspection of definitions.
\end{proof}
\subsection{Associative case} In the associative case we have more versions of the graph complex as in addition to the modular operads  $\Assud$ and $\Assod$ we have the intermediate operad $\KAssd$. Again, for the sake of simplicity of exposition we will only treat graphs without legs.

\begin{defi}A \emph{prestable ribbon graph} (or simply a \emph{prestable graph}) is a connected graph such that the set of half-edges around each vertex is partitioned into cyclically ordered subsets; additionally a pair of non-negative integers $(\gd,\bd)$ is assigned to each vertex. The integer $\gd$ is called the \emph{genus defect} and the integer $\bd$ is called the \emph{boundary defect} of the corresponding vertex. We do not allow both the genus defect and boundary defect to be zero at a $1$-valent vertex nor at a bivalent vertex whose two half-edges form a cycle.
\end{defi}
It is convenient to represent a vertex $v$ as a `ghost' surface $S_v$ having genus $\gd$ and $\bd$ free boundary components. Let $\Flag(v)=\coprod_{i\in I}\Flag_i$ be the given partition of $\Flag(v)$ into cyclically ordered subsets; then the remaining (non-free) boundary components of $S_v$ are labeled by $I$; moreover the open boundary intervals embedded into the $i$th boundary component are in one-to-one correspondence with the set $\Flag_i$ in a way compatible with the given cyclic ordering of $\Flag_i$.


\begin{figure}[h]
\[
\includegraphics[height=1.8in]{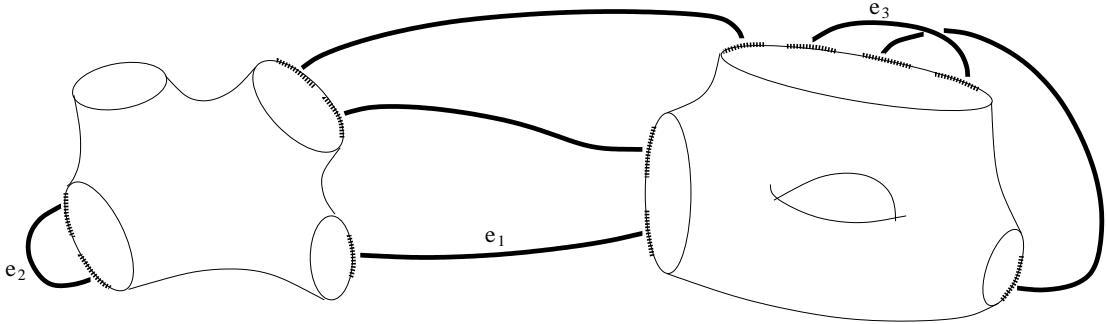}
\]
\caption{A prestable graph.\label{prestablegraph}}
\end{figure}
\FloatBarrier

This description is similar to that of a stable commutative graph complex; the ghost surfaces being analogous to the dotted loops. The notion of a prestable graph is nothing but a stable graph decorated by the operad $\OTFT$. The definition of a $d$-orientation is similar to the stable commutative case.
A $0$-orientation is again an ordering of edges, and a $1$-orientation is an ordering of vertices, an orientation of each edge \emph{and} for each vertex $v$, an ordering of the blocks $\Flag_i$ of the given partition of $\Flag(v)$.
As before a $d$-orientation is defined up to even permutation, so each prestable graph has two $d$-orientations, opposite to each other. A $d$-orientation on a prestable graph $G$ determines a nonzero element of
$\mathfrak{K}(G)\otimes\Det^d(G)\otimes\OTFT^d((G))$; see Remark~\ref{OTFTorientation}.

Just like a usual ribbon graph, a prestable graph $G$ has a set of \emph{boundary components}. These are defined as follows. Thicken the edges of $G$ to form ribbons connecting the ghost surfaces attached to the vertices. The boundary components of the resulting connected surface $S(G)$ will be called the boundary components of $G$; clearly their number is an isomorphism invariant of $G$. This generalizes the notion of a boundary component of a usual ribbon graph. Note that we can meaningfully speak of a sequence of edges of a prestable graph forming a boundary component. In particular, an edge (not necessarily a loop)  could itself form a boundary component.


\begin{figure}[h]
\[
\includegraphics[height=1.7in]{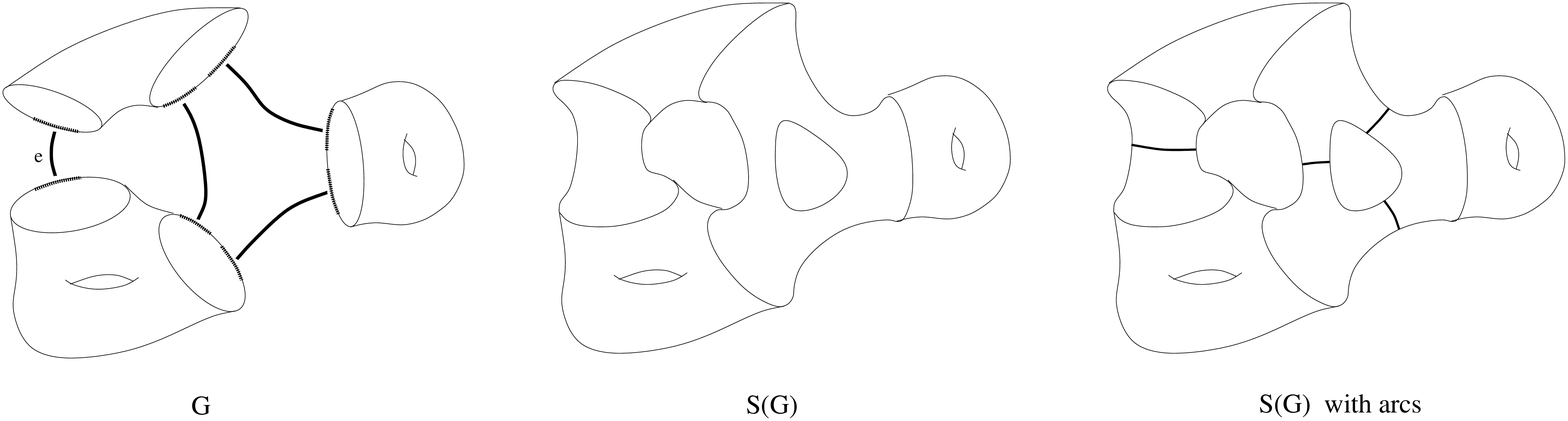}
\]
\caption{A prestable ribbon graph $G$ and its associated surface $S(G)$. Note that the edge $e$ itself forms a boundary component. \label{edgeisboundary}}
\end{figure}
\FloatBarrier

Further, the \emph{genus} of a prestable graph is defined as the genus of the surface $S(G)$.
We will now introduce complexes $\G_\bullet^{\Assod}(\gamma,\nu)$ for $d=0,1$. They will be called \emph{prestable ribbon} graph complex for $d=0$ and \emph{twisted prestable} graph complex for $d=1$. The underlying vector space for $\G_\bullet^{\Assod}(\gamma,\nu)$ is spanned by the isomorphism classes of $d$-oriented prestable graphs of genus $\gamma$ with $\nu$ boundary components; the grading is chosen in such a way that $\G_n$ corresponds to the graphs with $n$ edges. As in the commutative case, we identify an oriented graph with the negative of the oppositely oriented graph.
We will sometimes omit indicating explicitly the dependence of the complexes under considerations on the numbers $\gamma$ and $\nu$.

Let $G$ be a prestable graph and let $e$ be an edge connecting two vertices $v_1$ and $v_2$ (which could coincide). We will regard $e$ as a thin ribbon joining the surfaces $S_{v_1}$ and $S_{v_2}$; there results a new surface $S_{v_1v_2}$ which is the union of $S_{v_1}, S_{v_2}$ and $e$. Consider the new prestable ribbon graph $G_e$ obtained from $G$ by coalescing the vertices $v_1$ and $v_2$ into a new vertex decorated by $S_{v_1v_2}$. Then the differential in $\G_\bullet^{\Assod}$ is defined by the same formula as in the stable commutative case (\ref{dif}).

\begin{figure}[h]
\[
\includegraphics[height=1.8in]{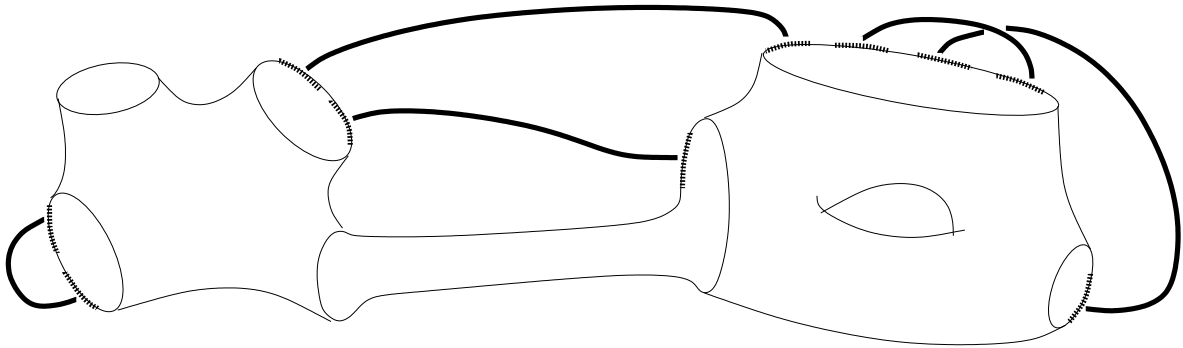}
\]
\caption{Graph obtained by `contracting' the edge $e$ in Figure~\ref{prestablegraph}.}
\end{figure}
\FloatBarrier

\begin{rem}

The main thrust of this section has been to give a purely combinatorial description, in terms of graphs, of certain (dual) Feynman transforms. Temporarily taking the opposite approach, replacing graphs by surfaces, we are led to arc systems and the point of view of Penner \cite{penner}.

We saw above that a prestable graph $G$ determines a surface $S(G)$ formed by replacing vertices by ghost surfaces and thickening edges into strips. Of course different prestable graphs may give rise to the same surface; one way to obtain a faithful representation is to endow $S(G)$ with a system of arcs, one for each thickened edge, connecting opposite boundaries. See Figure~\ref{edgeisboundary}.

This recipe gives rise to an equivalence between the category of prestable graphs and the category $\ArcCat(0)$ whose objects are
 compact oriented surfaces $G$ with boundary, equipped with a system $\Arc(G)$ of non-overlapping embedded intervals with endpoints in the boundary of $G$; each arc in $\Arc(G)$ is required to be non-isotopic (relative to the boundary) to any other arc in $\Arc(G)$, or to a boundary arc. Morphisms in $\ArcCat(0)$ are generated by homeomorphisms preserving orientation and arc systems, and by `arc deletions'.
We then have an isomorphism of chain complexes
$$\dft{\De}{\OTFT}^d((0))\cong\colim_{G\in\Iso\ArcCat(0)}\Det(\Arc(G))\otimes\De(H_1(S(G))),$$
where the differential on the right hand side is given by arc deletion.
 To reconstruct all of $\dft{\De}{\OTFT^d}$, including the operadic structure, we should use categories $\ArcCat(n)$ of arc systems in surfaces with $n$ open boundaries.
\end{rem}

We will now introduce two other types of associative graphs.
\begin{defi}
Let $G$ be a prestable graph. It is called a \emph{stable ribbon graph} if every vertex of $G$ has a zero boundary defect.
If every vertex of $G$ has zero genus defect, zero boundary defect and trivial partition on its half-edges, then  $G$ is called a \emph{ribbon graph}.
\end{defi}
\begin{rem} The term `ribbon graph' just introduced is indeed consistent with the common usage.
Indeed if each vertex of a prestable graph $G$ has zero genus defect, zero boundary defect and trivial partition on adjacent half-edges, then all its ghost surfaces are disks. Clearly these ghost surfaces could simply  be ignored, making $G$ a usual ribbon graph.
\end{rem}

Further we define two other associative graph complexes ${\G}_\bullet^{\KAss}(\gamma,\nu)$
and ${\G}_\bullet^{\Assu}(\gamma,\nu)$,
and their twisted versions ${\G}_\bullet^{\KAsso}(\gamma,\nu)$ 
and ${\G}_\bullet^{\Assuo}(\gamma,\nu)$; we will call these \emph{stable ribbon} graph complexes
 and \emph{ribbon} graph complexes respectively. The definitions of the underlying vector spaces are the same as in the prestable graph complexes. The differential is likewise defined just as in the prestable case, except if an operation of collapsing an edge produces a graph of a prohibited type then this operation is set to be zero. For example, the  prestable graph pictured in Figure~\ref{prestablegraph} is stable, but contraction of the loop $e_2$ yields a vertex with boundary defect $1$ and therefore does not contribute to the differential in ${\G}_\bullet^{\KAss}$. On the other hand, contracting the loop $e_3$ does yield another stable ribbon graph.

We have the following hierarchy of modular operads; both arrows are quotient maps:
$$\Assod\longrightarrow\KAssd\longrightarrow\Assud$$
This gives rise to quotient maps between the corresponding graph complexes:
$$\G^{\Assod}_\bullet(\gamma,\nu)\longrightarrow
\G^{\KAssd}_\bullet(\gamma,\nu)\longrightarrow
\G^{\Assud}_\bullet(\gamma,\nu).$$

We find it convenient to use the corresponding cochain complexes $\G^\bullet_\O(\gamma,\nu):=\Hom(\G_\bullet^\O(\gamma,\nu),\ground)$ and their cohomology $H^\bullet_\O(\gamma,\nu)$ where $\O$ is $\Assod$, $\KAssd$ 
or $\Assud$.   Finally we have the following analogue of Proposition \ref{stablecom} whose proof consists of simply unraveling the definitions.
\begin{prop}
Let $\O$ be $\Assod$, $\KAssd$
 or $\Assud$.
Then there is an isomorphism of complexes
\[\F^\bullet_{\De}\O((g,0))\cong
\bigoplus_{2(\gamma-1)+\nu=g-1,\enspace \nu>0} \G^\bullet_{\O}(\gamma,\nu)\]
\end{prop}
\noproof

\section{Examples}
To understand how Theorem \ref{dualalg} works in more down-to-earth terms, let us consider the case $\O=\Ass$, and let $V$ be a dg algebra with  a contracting homotopy $s:V\ra V$ and an invariant inner product $\langle, \rangle$ of degree $d$. We will now recall, in slightly different languange, the formulation of the dual construction given by Kontsevich in \cite{kontfeynman}.
Let $G$ be a $d$-oriented ribbon graph without legs.
Recall that a $1$-orientation means that the vertices are ordered and each edge is directed; reversing the order of two vertices or the direction of an edge results in changing the sign of~$G$. Similarly a $0$-orientation is an ordering of the edges, again with the understanding that an odd permutation means a reversal of sign. We associate the tensor $\langle v_1\ldots v_{n-1},v_{n}\rangle$ to any vertex of $G$ of valence $n$; this tensor can be viewed as an element in $\Pi^{d(n-1)}V^{\otimes n}$ by virtue of the isomorphism $\Pi^{d}V\cong V^*$. We then associate the `propagator' $\langle ?,s(?)\rangle\in \Pi^{d+1}(V^*\otimes V^*)$ to any edge of $G$ and perform contractions along all the edges of $G$. The resulting function $Z_V(G)$ on ribbon graphs, called the \emph{partition function} of $V$, is thus a cocycle on an appropriate version of the ribbon graph complex.

However, one needs to be careful to specify in what sort of ribbon graph complex a cycle is obtained. The original version of the (ribbon) graph complex defined by Kontsevich is what we have denoted by $\G_\bullet^{\Assud}$. It corresponds to the dual Feynman transform of the modular $\De$-operad $\Assud$ since the formula for the differential in this complex does not include contractions of loops. An algebra over this modular operad is \emph{not} an arbitrary differential graded Frobenius algebra, but one satisfying (\ref{rel1}) and (\ref{rel2}).

Similarly, a contractible dg Frobenius algebra satisfying (\ref{rel1}) determines a cocycle on the
stable ribbon graph complex $\G_{\bullet}^{\KAssd}$,
 and finally an arbitrary contractible dg Frobenius algebra determines a cocycle
 on the prestable graph complex $\G_{\bullet}^{\Assod}$.
\begin{example}\label{twodimensional}
Let $V$ be a two-dimensional $\Z/2$-graded vector space spanned by the elements $a$ and $1$ in degrees $1$ and $0$ respectively. We define an associative product on $V$ by requiring that $1$ be a two-sided unit and that $a^2=1$. There is an odd scalar product on $V$ given by $\langle a,1\rangle=1$. The differential $d$ is given by the formula $d(a)=1$ and $d(1)=0$. Finally, the contracting homotopy for $d$ has the form $s(1)=a$ and $s(a)=0$. Note that $\langle,\rangle^{-1}=a\otimes 1-1\otimes a$ and so (\ref{rel1}) is satisfied but  (\ref{rel2}) \emph{is not}.

Observe that the propagator $\langle ?,s(?)\rangle$ is zero on $1\otimes a$, $a\otimes a$ and $a\otimes 1$, so in the expression for the Feynman amplitude of a ribbon graph the non-zero contributions only come from the graphs having all their half-edges decorated by the element $1$. It follows that all vertices of such graphs should have odd valence. Consider the following ribbon graph with three vertices $A,B$ and $C$:\\
\[G=\xymatrix{*=0{\bullet}\ar^<A@{-}@/_-2pc/[rr]&*=0{\bullet}\ar@{-}[l]
\ar_<B@{-}[r]\@{-}@(ul, dl)\ar@{-}@(dl,ul)[]&*=0{\bullet}\ar_<C@{-}@/^2pc/[ll] }\]\\
Here the cyclic ordering around each vertex is given by a local embedding into the $2$-plane as shown in the picture above. We will first regard $G$ as being an element in
$\G_{\bullet}^{\Assuo}$,
 i.e. the usual graph complex where contracting loops is not allowed. The differential applied to $G$ produces a sum of four ribbon graphs. The two obtained by contracting the edge $AB$ or $BC$ have vertices of even valence and thus zero amplitude for $V$.  The other two are isomorphic \emph{and identically oriented} graphs of the following shape:\\
\[\xymatrix{*=0{\bullet}\ar@{-}@/_-1.5pc/[rr]&*=0{}\ar@{-}[l]
\ar@{-}[r]&*=0{\bullet}\@{-}@(ul, dl)\ar@{-}@(dl,ul)[]\ar@{-}@/^1.5pc/[ll] }\]\\
It follows that the amplitude corresponding to the linear combination of graphs $d(G)$ for the algebra $V$ is equal to $\pm 2$; the sign here depends on the orientation of $G$. Therefore $V$ \emph{does not} give a cocycle on the usual twisted ribbon graph complex $\G_{\bullet}^{\Assuo}$.

However we know that $V$ must give a cocycle in the twisted stable ribbon graph complex $\G_{\bullet}^{\KAsso}$; indeed one can check that the non-zero term described above cancels with the amplitude corresponding to the graph obtained by contracting the loop situated at the vertex $B$.
\end{example}

\begin{example}
The $\Z/2$-graded associative algebra underlying the contractible dga $V$ considered in the previous example gives rise to an involution in the super-Brauer group of $\ground$, i.e. the group of central simple $\Z/2$-graded associative $\ground$-algebra modulo Morita equivalence. If $\ground$ is algebraically closed this involution generates the super-Brauer group. It turns out that $V$ gives rise to a non-trivial cocycle on the twisted stable graph complex; the proof will be given elsewhere.
 On the other hand
the identity element is represented by the algebras
$\End(k^m\oplus(\Pi k)^n)$. It could be shown that any $\dft{\De}{\Assod}$-algebra structure supported on such a matrix algebra gives rise to the trivial cocycle on the twisted stable ribbon graph complex.
\end{example}

\begin{example}
Consider the two-dimensional contractible dga $V$ with an invariant odd scalar product, defined exactly as in Example~\ref{twodimensional}, except that the associative product is given by $a^2=0$. With this modification, both conditions (\ref{rel1}) and (\ref{rel2}) are satisfied; hence $V$ is an algebra over $\dft{\Det}{\Assuo}$. However since the corresponding Feynman amplitudes for all ribbon graph without legs are zero, the algebra $V$ does not furnish any nontrivial cocycles on the twisted ribbon graph complex.
\end{example}

\section{Moduli spaces of metric graphs}
In this section we introduce spaces of metric commutative and associative graphs and show that the algebraic graph complexes considered in section~\ref{stablegraphs} compute the homology
of various sheaves on these spaces.  For simplicity we treat graphs without legs throughout this section.
\subsection{Commutative case}

Consider the following functor from the opposite category $\Gamma^{op}((g,0))$ of stable (commutative) graphs without legs to the category $\TOP$ of topological spaces:
\[M^{\Commo}_g:G\mapsto\text{ the set of all functions $l:\Edge(G)\ra\mathbb{R}_{\geq 0}$ such that $\sum_{e\in\Edge(G)}l(e)=1$}.\]

\begin{defi}The \emph{moduli space of stable metric graphs of genus $g$} is the colimit of the functor $M^{\Commo}_g$. It will be denoted by $\Mo^{\Commo}_g$.
\end{defi}
It is clear that the space $\Mo^{\Commo}_g$ consists of conformal isomorphism classes of stable metric graphs (whose first homology group has rank $g$), i.e. stable graphs whose edges are supplied with a positive length function. Contracting an edge corresponds to setting its length to zero.
\begin{rem}
There is an open dense subspace $\Mo^{\Comm}_g$ in $\Mo^{\Commo}_g$ consisting of the usual metric graphs (i.e. those with vertices all of genus zero).
The space $\Mo^{\Comm}_g$ is a classifying space for $\operatorname{Out}(F_g)$,
 the group of outer automorphisms of the free group on $g$ generators, see \cite{CulV}. Allowing the contraction of loops leads to the compactification $\Mo^{\Commo}_g$.
\end{rem}

The space $\Mo^{\Commo}_g$ is an (orbi)-simplicial complex: it is a union of quotients of open topological simplices by actions of finite groups. There is precisely one simplex $\sigma_G$ for each each isomorphism class of $G\in\Gamma((g,0))$. Denote by $\st(\sigma_G)$ (the \emph{open star} of $\sigma_G$) the union of interiors of all simplices containing a given simplex $\sigma_G$. Clearly the collection $\{\st(\sigma_G)\}$ is an open covering of $\Mo^{\Commo}_g$.

Consider  a functor $F:\Gamma^{op}((g,0))\mapsto \dgVect$. Then $F$ gives rise to a dg sheaf ${\mathscr F}$ on $\Mo^{\Commo}_g$. Namely, we define ${\mathscr F}$ to be the sheaf associated with the presheaf
whose space of sections over $\st(\sigma_G)$ is
 $F(G)_{\Aut(G)}$. This sheaf is \emph{constructible}, i.e. its restriction onto the interior of each simplex $\sigma_G$ is isomorphic to a constant sheaf.

Now let $\O$ be a modular $\De$-operad. Let ${\mathscr F}_\O$ be the dg sheaf on
$\Mo^{\Commo}_g$ corresponding to the functor $F_\O:G\mapsto {\De}(G)\otimes\O((G))^*$.  Note that the Feynman transform $\FDeo((g,0))$ is a dg vector space augmented into $\O((g,0))^*$. 
 Denote by $\hat{\F}_{\De}\O((g,0))$ the dg vector space $\Pi\Ker(\F_{\De}\O((g,0))\ra \O((g,0))^*)$. In other words $\hat{\F}_{\De}\O((g,0))$ is obtained from the Feynman transform $\FDeo((g,0))$ by throwing away the term corresponding to the graph with no edges and then reversing parity.

\begin{theorem}\label{orbi}
There is an isomorphism
\[H\left(\Mo^{\Commo}_g,{\mathscr F}_{\O}\right)\cong H\left({\hat{\F}}_{\De}\O((g,0))\right).\]

\end{theorem}
\begin{rem}
The reason we need to consider ${\hat{\F}}_{\De}\O((g,0))$ rather than ${{\F}}_{\De}\O((g,0))$ is that in our simplicial framework a metric graph necessarily has at least one non-trivial edge.
\end{rem}
Before tackling the proof of Theorem \ref{orbi} observe the following special cases corresponding to the operads $\Commod$ or
$\Commud$. Denote by $\det^d$ the sheaf on $\Mo^{\Commo}$ corresponding to the modular $\De$-operad $\Commod$. Note that if $d=0$ then $\det^d$ is the constant sheaf $\ground$. By abuse of notation we will also denote by $\det^d$ the inverse image  (restriction) of $\det^d$ on $\Mo^{\Comm}_g\subset \Mo^{\Commo}_g$. For $\O=\Commod$ or $\O=\Commud$ define $\overline{H}^\bullet_{\O}(g)$  by the following formula:
\[\hat{H}^k_{\O}(g)=\begin{cases}H^{k-1}_{\O}(g) & \text{if } k>0\\
\ground & \text {if } k=0\end{cases}\]

\begin{cor}\label{orb}
There is a following isomorphism; here $H_c(-)$ denotes cohomology with compact supports:
\[H^\bullet_c(\Mo^{\O}_g, {\det}^d) \cong \hat{H}^\bullet_{\O}(g).\]
\end{cor}
\noproof
\begin{rem}
For uniformity we used cohomology with compact supports in the formulation of the above theorem for both moduli spaces. Note, however, that the space $\Mo^{\Commo}_{g}$ is itself compact and so the usual sheaf cohomology is the same as cohomology with compact supports.
\end{rem}

In preparation for the proof of Theorem \ref{orbi} we will introduce the notion of a \emph{labeled stable graph}. Note that the number of half-edges of a stable graph of genus $g$ is not greater than $6g-6$. Let $S$ be a finite set of cardinality not less than $6g-6$ and consider the category $\C_g$ of \emph{$S$-labeled stable graphs} (or simply labeled stable graphs)  which is defined as follows.

An object of $\C_g$ is a stable graph $G$ of genus $g$ without legs together with a injective map $\Flag(G)\hookrightarrow S$; thus the half-edges of $G$ are labeled by distinct elements of $S$. The morphisms in $\C_g$ are morphisms of stable graphs compatible with labellings, so they are generated by isomorphisms preserving labels and by edge-contractions.
Note that a labeled stable graph has no non-trivial automorphisms.
 The permutation group $A:=\Aut(S)$ acts on the category $\mathcal C_g$ by changing the labels of the half-edges of labeled stable graphs.


Consider the following functor from the opposite category $\mathcal C^{op}_g$ to the category $\TOP$:
\[F_g:G\mapsto\text{ the set of all functions $l:\Edge(G)\ra\mathbb{R}_{\geq 0}$ such that $\sum_{e\in\Edge(G)}l(e)=1$}.\]

\begin{defi}The \emph{moduli space of labeled stable metric graphs of genus $g$} is the colimit of the functor $F_g$. It will be denoted by $\MM_g$.
\end{defi}
It is clear that the space $\MM_g$ consists of conformal isomorphism classes of labeled stable metric graphs (whose first homology group has rank $g$), i.e. labeled stable graphs whose edges are supplied with a positive length function. Contracting an edge corresponds to setting its length to zero.

The space $\MM_g$ is a simplicial complex with precisely one simplex $\sigma_G$ for each isomorphism class of labeled stable graph $G\in\C_g$. The group $A$ acts on this simplicial complex properly discontinuously; the stabilizer of a simplex $\sigma_G$ is the group $\Aut(G)$ of automorphisms of the graph $G$.
The functor $\C_g\ra \Gamma((g,0))$ of forgetting the labelings on half-edges determines a map of topological spaces $f:\MM_g\ra \Mo^{\Commo}_g$, which induces a homeomorphism between the quotient $\MM_g/A$ and $\Mo^{\Commo}_g$.
 In other words, we have represented the space $\Mo^{\Commo}_g$ of stable metric graphs as a global quotient of a simplicial complex by the action of a finite group.

Denote by $\st(\sigma_G)$
 the union of interiors of all simplices containing a given simplex $\sigma_G$. Clearly the collection $\{\st(\sigma_G)\}$ is an open covering of $\MM_g$.

Consider  a functor $F:\C_g\ra \dgVect$. Then $F$ gives rise to a dg sheaf ${\mathcal F}$ on $\MM_g$. Namely, we define ${\mathcal F}$ to be the sheaf associated with the presheaf
whose space of sections over $\st(\sigma_G)$ is
 $F(G)$. This sheaf is \emph{constructible}, i.e. its restriction onto the interior of each simplex $\sigma_G$ is isomorphic to a constant sheaf.

Now consider a functor $\Gamma((g,0))^{op}\ra\dgVect$ and its restriction along the label-forgetting functor $\C_g\ra\Gamma((g,0))$. Associated to these functors are the sheaves $\mathscr F$ and $\mathcal F$ on $\Mo^{\Commo}_g$ and $\MM_g$ respectively. It follows directly from definitions that \begin{itemize}\item
$\mathcal F=f^{-1}\mathscr F$;
\item$\mathscr F=f^A_*\mathcal F$.\end{itemize}
Here $f^{-1}$ and  $f^A_*$ stand  for the inverse image and $A$-equivariant direct image functors respectively.
We can now prove Theorem \ref{orbi} following \cite{LazVor}, Theorem 3.7.
\begin{proof}[Proof of Theorem \ref{orbi}]
Note that the functors $f^{-1}$ and $Rf_*^A$ establish an equivalence of the derived category of sheaves on $\MM_g$ with  a full subcategory in the derived category of $A$-equivariant sheaves on $\Mo^{\Commo}_g$.  Therefore there is an isomorphism
\[R\Gamma(\Mo^{\Commo}_g,\mathscr F)\cong R\Gamma^A(\MM_g,\mathcal F)\]
where $R\Gamma^A(\MM_g,\mathcal F)=\left(R\Gamma(\MM_g, F)\right)^A$ is the $A$-equivariant derived global sections functor (observe that since $A$ is a finite group the functor of $A$-invariants is exact).
Furthermore, since $\mathcal F$ is a constructible sheaf its derived global sections could be computed using the \v{C}ech complex $\check{C}(\MM_g,\mathcal F)$ corresponding to the covering of $\MM_g$ by the open stars of its vertices. It only remains to note that the dg vector space of $A$-invariants $\check{C}^A(\MM_g,\mathcal F)$ is isomorphic to ${\hat{\F}}_{\De}\O((g,0))$.
\end{proof}

\subsection{Associative case}
The moduli spaces of metric ribbon graphs and their decorated versions are treated in a way completely parallel to the case of metric commutative graphs. Therefore we shall restrict ourselves with providing the relevant definitions and formulations. In fact, the most general treatment in the associative case would entail considering modular non-$\Sigma$-operads and corresponding sheaves on the moduli spaces of metric ribbon graphs. We won't go so far and instead consider the most important special cases related to the operad $\Ass$.

Let
$\Gamma((\gamma,\nu,0))$
 be the category whose objects are prestable ribbon graphs of genus $\gamma$ with $\nu$ boundary components and having no legs; the morphisms are generated by isomorphisms and edge-contractions.

Consider the following functors from the opposite category
 $\Gamma^{op}((\gamma,\nu,0))$
 of prestable graphs  to the category $\TOP$ of topological spaces:
\[M^{\overline{\Ass}}_{\gamma,\nu}:G\mapsto\text{ the set of all functions $l:\Edge(G)\ra\mathbb{R}_{\geq 0}$ and such that $\sum_{e\in\Edge(G)}l(e)=1$}.\]
\begin{defi}The \emph{moduli space of prestable metric ribbon graphs
} (or simply \emph{prestable metric graphs}) is the colimit of the functor $M^{\Asso}_{\gamma,\nu}$. It will be denoted by $\Mo^{\Asso}_{\gamma,\nu}$.
\end{defi}
Similarly to the commutative case the points in $\Mo^{\Asso}_{\gamma,\nu}$ are conformal isomorphism classes of stable metric ribbon graphs (of genus $\gamma$ with $\nu$ boundary components) i.e. those stable ribbon graphs whose edges are supplied with a positive length function $l$. There are certain natural subspaces in $\Mo^{\overline{\Ass}}_{\gamma,\nu}$.
\begin{defi}\
\begin{enumerate}
\item The \emph{moduli space of stable metric ribbon graphs} is the subspace $\Mo^{\KAss}_{\gamma,\nu}$ of $\Mo^{\Asso}_{\gamma,\nu}$ consisting of those prestable metric graphs
which have no vertices of positive boundary defect.
\item
The \emph{moduli space of metric ribbon graphs} is the subspace
$\Mo^{\Ass}_{\gamma,\nu}$ of $\Mo^{\Asso}_{\gamma,\nu}$ consisting of prestable metric graphs whose ghost surfaces are spheres with only one boundary component.
\end{enumerate}
\end{defi}
It is clear that the space $\Mo^{\Asso}_{\gamma,\nu}$  consists of conformal isomorphism classes of usual metric ribbon graphs of genus $\gamma$ and with $\nu$ boundary components i.e. graphs whose edges are supplied with a positive length function. Contracting an edge corresponds to setting its length to zero. Allowing the contraction of loops we arrive at two different notions of stability: if the contraction of boundary components is prohibited we obtain the notion of a stable metric ribbon graph; dropping this restriction we get prestable metric ribbon graphs. Note in passing that there are other moduli spaces corresponding to modular operads intermediate between $\Assu$ and $\Asso$; see Remark~\ref{operadzoo}.

We can now formulate the analogue of Theorem \ref{orbi} (or, rather, of Corollary \ref{orb})  in the case of associative graphs. The proof, which carries over from the commutative case almost verbatim, will be omitted.
For $\O=\Assud$, $\Assod$ or $\KAssd$  we define
\[\hat{H}^k_{\O}(\gamma,\nu)=\begin{cases}H^{k-1}_{\O}(\gamma,\nu) & \text{if } k>0\\
\ground & \text {if } k=0\end{cases}\]
\begin{theorem}\label{orbi1}
There is an isomorphism:
\[ H^\bullet_c(\Mo^{\O}_{\gamma,\nu}, {\det}^d) \cong
\hat{H}^\bullet_{\O}(\gamma,\nu).\]
\end{theorem}\noproof
\begin{rem}
Our combinatorial moduli spaces admit interpretations in terms of moduli spaces of algebraic curves over $\mathbb C$ as follows.
\begin{enumerate}\item Denote by $\Mo_{\gamma,\nu}$ the coarse (uncompactified) moduli space of curves of
genus $\gamma$ with $\nu$ marked points and by $\Delta^{n-1}$ the $(n-1)$-dimensional topological simplex. Then there is a homeomorphism $\left(\Mo_{\gamma,\nu}\times \Delta^{\nu-1}\right)/\S_\nu\cong\Mo^{\Ass}_{\gamma,\nu}$. Here $\S_\nu$ acts on $\Mo_{\gamma,\nu}$ by permuting the marked points, on $\Delta^{\nu-1}$ by permuting the barycentric coordinates.
This follows from the theory of Jenkins-Strebel differentials. \item Denote by $\overline{\Mo}_{\gamma,\nu}$
the moduli space of stable curves -- the Deligne-Mumford compactification of ${\Mo}_{\gamma,\nu}$. Two stable curves are \emph{equivalent}
if there is a homeomorphism between them that is complex-analytic on all components containing marked points. The
quotient $\K\Mo_{\gamma,\nu}$ of $\overline{\Mo}_{\gamma,\nu}$ by the closure of this equivalence relation is called the
Kontsevich compactification of ${\Mo}_{\gamma,\nu}$ and there is a homeomorphism $\Mo_{\gamma,\nu}^{\K\Ass}\cong \left(\K\Mo_{\gamma,\nu}\times \Delta^{\nu-1}\right)/\S_\nu$.
This is stated in \cite{Kon} and proved in \cite{Loo, Zv}.\item Consider the moduli space of stable curves of genus $\gamma$ with $\nu$ marked points
where each marked point is decorated by a non-negative number, its \emph{perimeter}. One requires that the sum of all perimeters equals $1$. This moduli space is clearly homeomorphic to $\overline{\Mo}_{\gamma,\nu}\times\Delta^{\nu-1}$.
Two stable curves are then \emph{equivalent} if there is a homeomorphism between them that is  complex analytic on all components having at least
one marked point with a non-zero perimeter. The quotient of $\overline{\Mo}_{\gamma,\nu}\times\Delta^{\nu-1}$ modulo the
closure of this equivalence relation is a compactification of the decorated moduli space of curves. The quotient of this compactification by $\S_\nu$ is homeomorphic to $\Mo^{\Asso}_{\gamma,\nu}$. This is proved in \cite{Mon}.
\end{enumerate}
\end{rem}

\section{BV-resolution of a modular operad}
Recall that algebras over the dual Feynman transform of a modular $\De$-operad $\O$ are essentially $\O$-algebras together with a choice of contracting homotopy. A natural question is whether one can adapt this construction to not necessarily contractible $\O$-algebras. In this section we will outline such a construction and explain how it provides a resolution of $\O$ closely related to the canonical resolution $\F_{\mathfrak{K}\otimes\D^*}\F_\D\O$. We will call our variation on the dual Feynman transform the \emph{Boardman-Vogt resolution} in this context because it is similar to the topological tree complex studied in \cite{BV}.
\subsection{Basic construction and description of algebras}
\begin{defi}Let $\D$ be a cocycle and $\O$ be a modular $\D$-operad. Its BV-resolution $\BVDo$ is the extended modular $\D$-operad freely generated over $\O_+$ by an odd operation $s$ and an even operation $t$, both in $ \BVDo((0,2))$, subject to the relations:\begin{itemize}\item
$s^2$=0;\item $t^2=t$;\item $st=ts=0$.
\end{itemize}
The differential $d$ on $\BVDo$ extends the differential on $\O$, and $d(s)=\mathbf{1}-t$ and  $d(t)=0$.
\end{defi}
\begin{rem}\label{special}
In keeping with our usual notational conventions from Section \ref{maincon} we can write $\left(\O[s,t]/(s^2,t^2-t,st)\right)_+$ for the BV-resolution of $\O$ (without the differential). There is an obvious surjection $\BVDo\ra\FvDo$ obtained by setting $t$ to zero. In particular, any $\FvDeo$-algebra automatically becomes a $\BVDeo$-algebra.
Furthermore, setting $s$ to zero we obtain an augmentation $\BVDo\ra \O_+$.  Taking in particular the case $\D=\De$, any $\O$-algebra $V$ gives canonically a $\BVDeo$-algebra. The operators $s$ and $t$ act on $V$ as the zero and identity operators respectively.
\end{rem}

We have the following analogue of Theorem \ref{dualalg}:
\begin{theorem}\label{dualalg1}
Let $\O$ be a modular $\De$-operad and $V$ be a dg vector space with an inner product $\langle,\rangle$ of degree $d$. Then the structure of an $\BVDeo$-algebra on $V$ is equivalent to the following data:
\begin{enumerate}\item
The structure of an $\O$-algebra on $V$.\item
An odd operator $s:V\ra V$ such that
\begin{itemize}\item
$s^2=0$;\item
$\langle s(a),b\rangle=(-1)^{|a|}\langle a,s(b)\rangle$.
\end{itemize}\item
An even operator $t:V\ra V$ such that
\begin{itemize}
\item $t^2=t$;
\item
$dt=td$;\item
$\langle t(a),b\rangle=\langle a,t(b)\rangle$.\end{itemize}
\item The following identities for the operators $s$ and $t$ hold:
\begin{itemize}\item
$st=ts=0$;\item
$(ds+sd)(a)=a-t(a)$ for any $a\in V.$
\end{itemize}
\end{enumerate}
\end{theorem}
\begin{proof}
This theorem follows directly from definitions; note that the invariance property of $t$ with respect to the inner product $\langle,\rangle$ has no sign since $t$ is an even operator.
\end{proof}

\begin{example}\label{Hodge}
An example of such a structure comes from any $\O$-algebra. Such an algebra $V$ is a dg vector space with an inner product $\langle, \rangle$; the differential is compatible with the inner product in the sense that $\langle da,b\rangle+(-1)^{|a|}\langle a,db\rangle=0$. Let $V_0=\Ker d$ and $U=\im d$. Choose a complement $W$ to $U$ inside $V_0$ so that $W\oplus U=V_0$. Then $U$ is an maximal isotropic subspace of $W^\perp$; choose an isotropic complement $U^\prime$. We have therefore
\[V=W\oplus(U\oplus U^\prime).\]
It is clear that the inner product $\langle, \rangle$ is nondegenerate on $W\cong H_\bullet(V)$ and the pairing between $U$ and $U^\prime$ is likewise nondegenerate.

Define the operator $t:V\ra V$ to be the projection onto $W$ and $t^\prime$ to be the projection onto $U\oplus U^{\prime}$. Set $s:=\left(d\mid_{U\oplus U^{\prime}}\right)^{-1}\circ t^\prime$; it is an analogue of the Green operator in Hodge theory.
Then it is easy to see that the $\O$-algebra $V$ supplied with the operators $s$ and $t$ satisfies the conditions of Theorem \ref{dualalg1} and so  it is an algebra over the BV-resolution of $\O$.

It is important to stress that the structure of a $\BVDo$-algebra on $V$ described here is different from one obtained via the augmentation $\BVDo\ra \O$ as in Remark \ref{special}. In particular, the operator $s$ is not zero on $V$ unless $V$ has vanishing differential.
\end{example}
\subsection{Stable BV-graph complexes}
In this subsection we will give another, more concrete, description of the BV-resolution of a modular $\D$-operad as a kind of decorated graph complex. We start by defining the categories $\Gamma_{\BV}((g,n))$ of \emph{two-colored stable graphs} or \emph{BV-graphs}.

\begin{defi} A BV-graph is an extended stable graph with the following additional structure:
the set $\Edge(G)$ is partitioned into two subsets: those of black edges and of white edges. We also require that the unstable vertices, i.e. the bivalent vertices of genus $0$, only occur next to legs.
\end{defi}

 We will denote the set of black edges of a BV-graph $G$ by $\Edge_b(G)$. For typographic reasons we will draw the black edges using straight lines and the white edges using wiggly lines. Note that the \emph{legs} of a BV-graph have no color; to emphasize this we will draw the legs using dotted lines.

There is an obvious notion of isomorphism between two BV-graphs. The objects of $\Gamma_{\BV}((g,n))$ are BV-graphs of genus $g$ with $n$ labeled legs. There are three types of morphisms in $\Gamma_{\BV}((g,n))$ which generate all morphisms:\begin{enumerate}\item
isomorphisms of BV-graphs preserving labelings.\item
contractions of black edges. For a black edge $e\in\Edge_b(G)$ this operation will be written as $G\mapsto G_e$.\item
replacing a black edge by a white edge. For a black edge $e\in\Edge_b(G)$ this operation will be written as $G\mapsto G^e$.\end{enumerate}

We can now introduce the two-colored graph version of the BV-resolution of a modular $\D$-operad $\O$ which will be temporarily  denoted by $\BV^\prime_{\D}(\O)$.
 Given a BV-graph $G$ we put
 $$\O[G]:=\Det(\Edge_b(G))\otimes{\D}(G)\otimes\O((G)).$$
 Here the space $\O((G))$ of $\O$-decorations on $G$ and the twisting
 $\D(G)$ are defined by forgetting the coloring on $G$ and just regarding it as an extended stable graph.  Let  $e\in \Edge_b(G)$ be a black edge. Then the contraction $G\mapsto G_e$ determines a parity-reversing linear map
$d_e:\O[G]\ra\O[G_e]$ given by the operadic composition in $\O$; similarly the operation $G\mapsto G^e$ determines tautologically a map $d^e:\O[G]\ra\O[G^e]$.

We define  \[\BVDotemp((g,n)):=\colim_{G\in\Iso\Gamma_{\BV}((g,n))} \O[G],\]
with differential $d$ determined by the formula
\begin{equation}\label{diff1}d|_{\O[G]}=d_\O+\sum_{e\in\Edge_b(G)}[d_e+d^e],\end{equation}
where $d_\O$ is the internal differential induced by the differential on $\O$.
\begin{prop}\label{twoBVs}
There is an isomorphism of complexes
\[\BVDo((g,0))\cong\BVDotemp((g,0)).\]
The operadic composition in $\BVDo$ is given by glueing decorated BV-graphs according to the following rule: graft legs to make a new edge and then contract the newly formed edge, using the operadic composition in $\O$.
\end{prop}
\begin{proof}
Using Lemma \ref{key} we see that the space $\O[s,t]$ is a space
of twisted $\O$-decorated extended stable graphs with \emph{two types of edges}: black edges corresponding to $s$ and the white edges corresponding to $t$.

The relation $s^2=0$ is interpreted pictorially as $\xymatrix{\ar@{.}[r]&*=0{\bullet}\ar@{-}[r]&*=0{\bullet}\ar@{-}[r]&*=0{\bullet}
\ar@{.}[r]&}=0$; thus graphs having an unstable bivalent vertex separating two black edges are zero. Similarly the relation $t^2=t$ can be rewritten as
\[\xymatrix{\ar@{.}[r]&*=0{\bullet}\ar@{~}[r]&*=0{\bullet}\ar@{~}[r]&*=0{\bullet}
\ar@{.}[r]&
}=\xymatrix{
\ar@{.}[r]&*=0{\bullet}\ar@{~}[r]&*=0{\bullet}\ar@{.}[r]&}.\] In other words the unstable bivalent vertices separating two white edges could simply be forgotten. Finally $st=ts=0$ is equivalent to
\[\xymatrix{\ar@{.}[r]&*=0{\bullet}\ar@{-}[r]&*=0{\bullet}\ar@{~}[r]&*=0{\bullet}
\ar@{.}[r]&
}=\xymatrix{
\ar@{.}[r]&*=0{\bullet}\ar@{~}[r]&*=0{\bullet}\ar@{-}[r]&*=0{\bullet}
\ar@{.}[r]&}
=0.\]
This means that the unstable bivalent vertices only occur next to legs. Next, since $d(t)=0$  the formula for the differential only involves the $s$-edges (i.e. the black edges); finally the formula $d(s)=\mathbf{1}-t$ precisely corresponds to equation (\ref{diff1}).
\end{proof}
\begin{rem}Note that the subspace in $\BVDo((g,n))=\BVDotemp((g,n))$ corresponding to $\O$-decorated graphs having \emph{at least one white edge} is closed with respect to the differential. The corresponding quotient dg vector space consists of $\O$-decorated graphs having only black edges, i.e. the usual commutative stable graphs. In the formula for the differential the term $d^e$ goes away and we are left with the usual graph differential as in the $\O$-graph complex (or the dual Feynman transform of the modular $\D$-operad $\O$). We obtain a surjective homomorphism
$\BVDo((g,n))\ra\FvDo((g,n))$; this is the same map as the one described in Remark \ref{special}.
\end{rem}

We conclude by explaining how $\BVDo$ is closely related to the standard free resolution of $\O$ provided by the twice-iterated Feynman transform.
Let $\Gamma_\bv((g,n))$ be the subset of $\Gamma_\BV((g,n))$ containing BV-graphs without unstable vertices. To any $G\in\Gamma_\bv((g,n)$ we can associate the BV-graph $G_t$ obtained from $G$ by glueing a white edge onto each leg. Then the image of the map $G\mapsto G_t$ consists of the BV-graphs all of whose legs are connected to white edges by unstable bivalent vertices.

We define $\bvDo$ to be the truncation of $\BVDo$ by the idempotent $t$, i.e. the suboperad of $\BVDo$ defined by
$$\bvDo((g,n)):=\colim_{G\in\Iso\Gamma_{\bv}((g,n))} \O[G_t].$$ The element $t$ serves as an operad unit for $\bvDo$.

\begin{prop} Let $\D$ be a cocyle, and let $\O$ be a modular $\D$-operad.
\begin{enumerate}
\item
We have an isomorphism
$\bvDo\cong\left(\M_{\D}\M_{\mathfrak{K}\otimes\D}\O\right)_+$ of extended modular $\D$-operads in $\Z/2$-graded vector spaces (forgetting the differential). Moreover, provided $\O((g,n))$ is finite dimensional for all $g$ and $n$, the extended modular $\D$-operad $\bvDo$ is isomorphic to the double Feynman transform
$\left(\F_{\mathfrak{K}\otimes\D^*} \F_\D \O\right)_+$.
\item The inclusion $\bvDo\hookrightarrow\BVDo$ is a quasi-isomorphism of extended modular $\D$-operads.
\item
The embedding $\O_+\hookrightarrow\BVDo$ and the augmentation
$\BVDo\ra \O_+$ are quasi-isomorphisms of extended modular operads.
\end{enumerate}
\end{prop}

\begin{proof} We first note that $\bvDo$ and $\left(\M_{\D}\M_{\mathfrak{K}\otimes\D}\O\right)_+$ are isomorphic as (extended) stable $\S$-modules; this is the content of Lemma 5.5 of \cite{GeK}. The operadic compositions match up, because the usual glueing of BV-graphs $G$ and $G'$, in which grafted legs form a white edge, corresponds to the glueing of $G_t$ and $G'_t$ described in  Proposition~\ref{twoBVs} (the same applies to self-glueings). If each $\O((g,n))$ is finite dimensional, then it is straightforward to check that the differentials on
$\F_{\mathfrak{K}\otimes\D^*} \F_\D \O \cong \M_{\D}\M_{\mathfrak{K}\otimes\D}\O$ and $\bvDo$ are the same.

Next we turn to the statement that
$\bvDo((g,n))\hookrightarrow\BVDo((g,n))$ is a quasi-isomorphism. For $(g,n)=(0,2)$ this is clear: the inclusion of $\ground\cdot t$ into the dg vector space  $W=\ground\cdot 1\oplus\ground\cdot t\oplus\ground\cdot\mathbf{1}$ with $d(s)=\mathbf{1}-t$ and $d(t)=0$, is indeed a quasi-isomorphism.
When $(g,n)\neq(0,2)$, we need to consider for each $G\in\Gamma_\bv((g,n))$ the $3^n$ BV-graphs obtained by tacking onto each leg of $G$ either a white edge, a black edge or nothing; choosing white edges for all legs produces $G_t$.
Now it all boils down to showing that the inclusion $(\ground\cdot t)^{\otimes n} \hookrightarrow
W^{\otimes n}$ is a quasi-isomorphism, which is obvious.

Finally, the composition
$\bvDo\hookrightarrow\BVDo\ra\O_+$ is a quasi-isomorphism, by Theorem 5,4 of \cite{GeK}. Since the composition $\O_+\hookrightarrow\BVDo\ra\O_+$ is the identity, the last statement of the proposition now follows.
\end{proof}


\begin{rem}
Let $V$ be an algebra over a modular $\De$-operad $\O$. According to
Example~\ref{Hodge}, a choice of Hodge decomposition on $V$ leads to a $\BVDeo$-algebra structure on $V$. Now we may truncate with the idempotent $t$, inducing a $\bvDeo$-algebra structure on $tV=H(V)$. It is tempting to regard $H(V)$ equipped with this `homotopy $\O$-algebra structure' as an explicit `minimal model' for $\O$.

In the special case that $V$ is a contractible $\O$-algebra, we obtain a $\bvDeo$-algebra structure on $H(V)=0$. This provides a slightly different take on the dual construction. Indeed, $\bvDeo$-algebra structures on the zero vector space are in one-to-one correspondence with cocycles on $\bvDeo(({\geq0}))/\langle\bvDeo(({>0}))\rangle\cong\FvDeo((0))$, and two algebra structures are quasi-isomorphic if and only if the corresponding cocycles are cohomologous. Thus
 the dual construction applied to a contractible $\O$-algebra could be viewed as the identification of its minimal model.

In this connection we mention the following, apparently intractable question: given a modular $\De$-operad $\O$, can every cohomology class on $\FvDeo((0))$ be realized by a contractible $\O$-algebra? Note that for two quasi-isomorphic operads their homotopy categories of algebras are isomorphic. If the corresponding statement were true for \emph{modular $\De$-operads} then applying it to modular $\De$-operads $\O$ and $\bvDeo$ and a given cocycle on $\FvDeo((0))$ considered as a $\bvDeo$-structure on the zero vector space we would get a positive answer to the posed question. However, even the \emph{correct definition} of homotopy categories of algebras over modular operads is not clear: we saw that a contractible (or even zero) $\O$-algebra does not necessarily have to be regarded as being trivial.
\end{rem}

\end{document}